\documentclass[10pt, reqno, oneside]{amsart}   
\usepackage{geometry,amsmath,amssymb,verbatim,tikz,tikz-cd,ytableau,hyperref,shuffle,cleveref}

\definecolor{darkred}{rgb}{0.7,0,0}

\newtheorem{thm}{Theorem}[section]
\newtheorem{lemma}[thm]{Lemma}
\newtheorem{prop}[thm]{Proposition}
\newtheorem{cor}[thm]{Corollary}

\newtheorem{definition}[thm]{Definition}

\theoremstyle{definition}
\newtheorem{example}[thm]{Example}
\newtheorem{remark}[thm]{Remark}

\def\Hilb{\operatorname{H}}
\def\symmfnH{{\mathcal{H}}}
\def\symmfnE{{\mathcal{E}}}
\def\spn{\operatorname{span}}
\def\Sym{\operatorname{Sym}}
\def\trace{\operatorname{Tr}}
\def\sh{\operatorname{sh}}
\def\chr{\operatorname{char}}

\def\sgn{\operatorname{sgn}}

\def\CC{{\mathbb{C}}}

\def\QQ{{\mathbb{Q}}}
\def\RR{{\mathbb{R}}}
\def\ZZ{{\mathbb{Z}}}
\def\NN{{\mathbb{N}}}

\def\symm{{\mathfrak{S}}}

\def\aa{{\mathbf{a}}}
\def\bb{{\mathbf{b}}}
\def\cc{{\mathbf{c}}}
\def\dd{{\mathbf{d}}}

\def\gg{{\mathbf{g}}}
\def\kk{{\mathbf{k}}}
\def\vvv{{\mathbf{v}}}
\def\xx{{\mathbf{x}}}

\newcommand{\multichoose}[2]{\ensuremath{\left(\kern-.2em\left(\genfrac{}{}{0pt}{}{#1}{#2}\right)\kern-.2em\right)}}

\newcommand{\rvline}{\hspace*{-\arraycolsep}\vline\hspace*{-\arraycolsep}}

\title[Invariant theory for wreath products acting on superpolynomials]{Invariant theory for wreath products acting on superpolynomials}
\author{Trevor Karn}
\email{karnx018@umn.edu}
\author{Victor Reiner}
\email{reiner@umn.edu}
\address{School of Mathematics, University of Minnesota}
\begin{document}

\begin{abstract}
This paper considers a finite group $G$ acting linearly on the variables $V$ of a polynomial algebra, or an exterior algebra, or a superpolynomial algebra with both commuting and
anticommuting variables.
In this setting, the Hilbert series for the $G$-invariant subalgebra
turns out to determine the analogous Hilbert series for the
wreath product $P[G]$ acting on $V^n$ for any permutation group $P$ inside the symmetric group $\symm_n$ on $n$ letters. 

This leads to a structural result:  one can collate the direct sum for all $n$ of the
$\symm_n[G]$-invariant subalgebras 
to form a graded ring via an external shuffle product,
whose structure turns out to be
a superpolynomial algebra generated by the $G$-invariants.  A parallel statement holds for the direct sum of all $\symm_n[G]$-antiinvariants, which forms a graded ring via an external signed shuffle product, 
isomorphic to the 
superexterior algebra generated by the $G$-invariants. 
\end{abstract}

\maketitle

\section{Introduction}
\label{sec:intro}

Classical invariant theory of finite groups (see, e.g., Benson \cite{Benson}, Smith \cite{Smith}, Stanley \cite{Stanley-invariants}) starts with
a $\kk$-vector space $V \cong \kk^n$ having basis $x_1,\ldots,x_n$,
and finite subgroup $G$ of $GL(V)$, acting on the {\it symmetric
algebra} $\Sym V \cong \kk[x_1,\ldots,x_n]$. 
It studies the 
structure of
the {\it $G$-invariant subalgebra}
$$
(\Sym V)^G:=\{ f \in \Sym V: g(f(\xx))=f(\xx) \text{ for all }g \in G\},
$$
and more generally, for
any homomorphism $\chi: G \rightarrow \kk^\times$,
the {\it $\chi$-relative invariants of $G$}
$$
(\Sym V)^{G,\chi}:=
\{ f(\xx) \in \Sym V: g(f(\xx))=\chi(g) \cdot f(\xx) \text{ for all }g \in G\}.
$$
It also studies $G$-invariants and $\chi$-relative invariants for the
$G$-action on the {\it exterior algebra}
$\wedge V$, and the ``polynomial tensor exterior" algebra 
$\Sym V \otimes \wedge V$;
see \cite[Chap.~5]{Benson}, \cite[Chap.~9]{Smith}.
When $|G|$ lies in $\kk^\times$, so that $G$ acts semisimply, the theory benefits from versions of {\it Molien's Theorem} predicting the {\it Hilbert series} for the invariants and relative invariants, as averages over $G$.

The goal of this paper is to generalize and build on observations of Solomon \cite{Solomon-partition-identities} and Dias and Stewart \cite{DiasStewart} 
in this context, regarding Hilbert series
for the {\it wreath products} $G \wr \symm_n = \symm_n[G]$ of the {\it symmetric group} $\symm_n$ with $G$, acting on $V^n=V\oplus \cdots \oplus V$ and $\Sym(V^n)$; see
Section~\ref{sec:wreath-Molien}
for the definition of this action.  Dias and Stewart \cite[\S6.2]{DiasStewart} showed
that, for any permutation subgroup $P$ of $\symm_n$, the
Hilbert series of $(\Sym V)^G$ determines that of $(\Sym V^n)^{P[G]}$ in a simple fashion,
while Solomon \cite[\S3]{Solomon-invariants} summed the Hilbert series of $(\Sym V^n)^{\symm_n[G]}$ for all $n=0,1,2,\ldots$ into an elegant product generating function.
Their results and methods anticipate our first two results, Theorems~\ref{thm:wreath-Molien-theorem}, \ref{thm:compiled-wreath-hilbs} below, extending their statements to the following
more general context. 

Assume that the field $\kk$ is not of characteristic two, and consider a $\ZZ_2$-graded $\kk$-vector space $V=V_{\bar{0}} \oplus V_{\bar{1}}$, with $\kk$-bases $x_1,\ldots,x_{r_0}$ for $V_{\bar{0}}$, and  $\theta_1,\ldots,\theta_{r_1}$ for $V_{\bar{1}}$.  We think of
the $\{ x_i\}$ as commuting (``bosonic") polynomial variables,
and the $\{\theta_j\}$ as anticommuting (``fermionic") exterior variables, in the {\it superpolynomial algebra} having these various descriptions:
\begin{align*}
\Sym_\pm V
&\cong \Sym(V_{\bar{0}}) \otimes \wedge(V_{\bar{1}})\\
&\cong \kk[x_1,x_2,\ldots,x_{r_0}] \otimes
\wedge(\theta_1,\theta_2,\ldots,\theta_{r_1})\\
&\cong T(V)/I
\end{align*}
where $T(V)=\bigoplus_{d=0}^\infty T^d(V)$
with $T^d(V)=V^{\otimes d}$ the {\it tensor algebra}, and $I$ is the two-sided ideal of $T(V)$ generated by
the elements  
$a\otimes b-(-1)^{\bar{i} \cdot \bar{j}} \, b \otimes a$
where 
$a \in V_{\bar{i}}$ and $b \in V_{\bar{j}}$
for $\bar{i},\bar{j} \in \ZZ_2$.

Say that a finite subgroup $G$ of $GL(V)$ {\it preserves the $\ZZ_2$-grading} if $G \subseteq {GL(V_{\bar{0}}) \times GL(V_{\bar{1}})}$.
The algebra $\Sym_\pm V$ additionally
has a grading by $\ZZ \times \ZZ=\ZZ^2$
preserved by such groups $G$, defined as follows: for $(i,j) \in \NN^2$ with $\NN =\{0,1,2,\ldots\}$, let
$$
(\Sym_\pm V)_{ij}:=\Sym^i(V_{\bar{0}}) \otimes \wedge^j(V_{\bar{1}})
=\kk[x_1,x_2,\ldots,x_{r_0}]_i \otimes \wedge^j(\theta_1,\theta_2,\ldots,\theta_{r_1}).
$$
Thus the $G$-invariant subalgebra
$(\Sym_\pm V)^G$ has a $\ZZ^2$-graded {\it Hilbert series} in $\ZZ[[q,u]]$ defined by
\begin{equation}
\label{eq:G-invariants-hilb}
\Hilb((\Sym_\pm V)^G,q,u)
:=\sum_{(i,j) \in \NN^2} a_{ij} q^{i} u^{j},
\quad \text{ where }a_{ij}:=\dim_\kk \left( \Sym^i(V_{\bar{0}}) \otimes \wedge^j(V_{\bar{1}}) \right)^G.
\end{equation}
One can then let the wreath product $G \wr \symm_n=\symm_n[G]=\symm_n \ltimes G^n$ act on $V^n$, with
$G^n$ acting independently in each summand of $V^n$ and the subgroup $\symm_n$ permuting the summands; see Section~\ref{sec:wreath-Molien} for more details.
In this setting, any permutation subgroup $P$ of $\symm_n$ inherits an action on $V^n$,
and on $\Sym_\pm V^n$.  One can ask for
the $\ZZ^2$-graded Hilbert series of either the $P[G]$-invariants $(\Sym_\pm V^n)^{P[G]}$, 
or the $P[G]$-antiinvariants $(\Sym_\pm V^n)^{P[G],\sgn}$, where $\sgn$ is this composite homomorphism:
$$
\sgn:P[G] \twoheadrightarrow P \hookrightarrow \symm_n \overset{\sgn}{\longrightarrow} \{\pm 1\}.
$$
It turns out that these Hilbert series are {\it plethystic evaluations} of
the {\it cycle index} and {\it signed cycle index} of $P$,
which are these
elements of the {\it ring of symmetric functions} $\Lambda=\Lambda_\QQ=\QQ[p_1,p_2,\ldots]$:
\begin{align}
\label{eq:cycle-index-defn}
Z_P&:=\frac{1}{|P|}
\sum_{\sigma \in P} p_{\lambda(\sigma)}
=
\frac{1}{|P|}
\sum_{\sigma \in P} \prod_{\substack{\text{ cycles }\\C \text{ of }\sigma}} p_{|C|},\\
\label{eq:signed-cycle-index-defn}
Z^{\sgn}_P&:=\frac{1}{|P|}
\sum_{\sigma \in P} \sgn(\sigma) \cdot p_{\lambda(\sigma)}
=
\frac{1}{|P|}
\sum_{\sigma \in P} \prod_{\substack{\text{ cycles }\\C \text{ of }\sigma}} (-1)^{|C|-1} p_{|C|},
\end{align}
where $\lambda(\sigma)$ is the {\it cycle type partition} of $\sigma$.  
For a fixed power
series $f(q,u) \in \QQ[[q,u]]$,
the operation of {\it plethystic substitution} $h \mapsto h[f(q,u)]$ is the unique
homomorphism $\Lambda=\QQ[p_1,p_2,\ldots] \rightarrow \QQ[[q,u]]$ sending $p_r \mapsto f(q^r,u^r)$ for $r=1,2,\ldots$.
After deriving a straightforward extension of
Molien's Theorem to $\Sym_\pm V$ in Section~\ref{sec:super-Molien}, the following 
result is proven 
in 
Section~\ref{sec:wreath-Molien}.

\begin{thm}
\label{thm:wreath-Molien-theorem}
In the above setting, when $|G|,|P|$ lie in $\kk^\times$, then
\begin{align}
\label{eq:wreath-product-Molien}
\Hilb((\Sym_\pm V^n)^{P[G]},q,-u)
&=Z_P\left[\, \Hilb((\Sym_\pm V)^G,q,-u)\, \right]\\
\label{eq:wreath-product-sign-Molien}
\Hilb((\Sym_\pm V^n)^{P[G],\sgn},q,-u)
&=Z^{\sgn}_P\left[\, \Hilb((\Sym_\pm V)^G,q,-u)\, \right].
\end{align}
\end{thm}

Section~\ref{sec:compiled-hilb-section}
compiles the cases where $P=\symm_n$ for $n=0,1,2,\ldots$ into generating functions as follows.

\begin{thm}
\label{thm:compiled-wreath-hilbs}
In the above setting, when $\chr{\kk}=0$, one has these generating functions:
\begin{align}
\label{eq:hilb-of-shuffle-algebra}
\sum_{n=0}^\infty
t^n 
\Hilb((\Sym_\pm V^n)^{\symm_n[G]},q,u)
&=
\displaystyle
\prod_{\substack{(i,j) \in \NN^2:\\j \text{ odd}}}
(1+t q^{i} u^{j} )^{a_{ij}}
\bigg/
\displaystyle
\prod_{\substack{(i,j) \in \NN^2:\\j \text{ even}}}
(1-t q^{i} u^{j} )^{a_{ij}}\\
\label{eq:hilb-of-signed-shuffle-algebra}
\sum_{n=0}^\infty
t^n 
\Hilb((\Sym_\pm V^n)^{\symm_n[G],\sgn},q,u)
&=
\displaystyle
\prod_{\substack{(i,j) \in \NN^2:\\j \text{ even}}}
(1+t q^{i} u^{j} )^{a_{ij}}
\bigg/
\displaystyle
\prod_{\substack{(i,j) \in \NN^2:\\j \text{ odd}}}
(1-t q^{i} u^{j} )^{a_{ij}}.
\end{align}
\end{thm}

The form of the two products on the right sides in
\eqref{eq:hilb-of-shuffle-algebra}, \eqref{eq:hilb-of-signed-shuffle-algebra} above
are suggestive of our main result, Theorem~\ref{thm:shuffle-algebra-is-superpolynomial} below, which is proven in Section~\ref{sec:shuffles} and explained here.
It is natural to view
$(\Sym_\pm V)^G$ as a $\ZZ_2$-graded
$\kk$-vector space $U$ as follows:
\begin{equation}
\label{eq:ensuing-Z2-grading}
U:=(\Sym_\pm V)^G
= \underbrace{
\bigoplus_{\substack{(i,j) \in \NN^2:\\j\text{ even}}}
(\Sym_\pm V)_{ij}^G}_{U_{\bar{0}}
} 
\quad \oplus \quad
\underbrace{\bigoplus_{\substack{(i,j) \in \NN^2:\\j\text{ odd}}}
(\Sym_\pm V)_{ij}^G}_{U_{\bar{1}}
}. 
\end{equation}
Comparing this with the right side of \eqref{eq:hilb-of-shuffle-algebra}, one sees that the latter product is the $\ZZ^3$-graded Hilbert series $\Hilb(\Sym_\pm U, t,q,u)$ of the superpolynomial algebra $\Sym_\pm U =\Sym_\pm \left( (\Sym_\pm V)^G \right)$, in which
$$
a_{ij}=\dim_\kk (\Sym_\pm U)_{1,i,j}=\dim_\kk U_{ij}=\dim_\kk (\Sym_\pm V)_{ij}^G
$$
appears as the coefficient of $tq^i u^j$.
To similarly interpret the right side of \eqref{eq:hilb-of-signed-shuffle-algebra}, one can consider the {\it superexterior algebra on $U$} (see Manin \cite[\S3.2.6]{Manin}), expressible as a quotient
$$
\wedge_\pm U :=T(U)/J
$$
where $T(U)$ is the tensor algebra and $J$ is the two-sided ideal generated by all 
$a \otimes b+(-1)^{\bar{i} \cdot \bar{j}} \, b \otimes a$ where 
$a \in V_{\bar{i}}$ and $b \in V_{\bar{j}}$
for $\bar{i},\bar{j} \in \ZZ_2$.
One then sees that product on the right side of \eqref{eq:hilb-of-signed-shuffle-algebra} is the
$\ZZ^3$-graded Hilbert series $\Hilb(\wedge_\pm U,t,q,u)$
for $U=(\Sym_\pm V)^G=U_{\bar{0}} \oplus U_{\bar{1}}$ as in \eqref{eq:ensuing-Z2-grading}. 
The form of the left sides of \eqref{eq:hilb-of-shuffle-algebra}, \eqref{eq:hilb-of-signed-shuffle-algebra} then motivate defining algebra structures on these $\ZZ^3$-graded $\kk$-vector spaces:
\begin{align}
 \label{eq:G-shuffle-algebra-defn}  
R_G&:=\bigoplus_{n=0}^\infty
(R_G)_n \text{ where }(R_G)_n:=
(\Sym_\pm V^n)^{\symm_n[G]},\\
&=\bigoplus_{(n,i,j) \in \NN^3}
(R_G)_{nij} \text{ where }(R_G)_{nij}:=
(\Sym_\pm V^n)_{ij}^{\symm_n[G]},\\
\label{eq:G-signed-shuffle-algebra-defn} 
R^{\sgn}_G&:=\bigoplus_{n=0}^\infty (R^{\sgn}_G)_n \text{ where }
(R^{\sgn}_G)_n :=(\Sym_\pm V^n)^{\symm_n[G],\sgn}\\
&=\bigoplus_{(n,i,j) \in \NN^3}
(R^{\sgn}_G)_{nij} \text{ where }(R_G)_{nij}:=
(\Sym_\pm V^n)_{ij}^{\symm_n[G],\sgn}.
\end{align}
Section~\ref{sec:shuffles} 
defines a graded (external) {\it shuffle product} $\shuffle:(R_G)_a \otimes (R_G)_b \rightarrow (R_G)_{a+b}$ on $R_G$, and 
{\it signed shuffle product}  $\shuffle^{\sgn}: (R^{\sgn}_G)_a \otimes (R^{\sgn}_G)_b \rightarrow (R^{\sgn}_G)_{a+b}$ on $R_G^{\sgn}$, in order to prove the following.

\begin{thm}
\label{thm:shuffle-algebra-is-superpolynomial}
In the context of Theorem~\ref{thm:compiled-wreath-hilbs},
one has $\ZZ^3$-graded algebra isomorphisms
\begin{align}
\label{eq:shuffle-algebra-isomorphism}
\Sym_\pm \left( (\Sym_\pm V)^G \right)
& \cong R_G \text{ with its }\shuffle\text{ product},\\
\label{eq:signed-shuffle-algebra-isomorphism}
\wedge_\pm\left( (\Sym_\pm V)^G \right)
& \cong R^{\sgn}_G \text{ with its  }\shuffle^{\sgn}\text{ product}.
\end{align}
\end{thm}

We close this introduction with some
examples and further context.
\begin{example}
\label{ex:motivating-Thibon-example}
An example which arose in work of the first author \cite{Karn-thesis} concerns the anticommutative case, where 
$V=V_{\bar{1}}=\kk^r$, so that $\Sym_\pm V = \wedge V =\wedge(\theta_1,\ldots,\theta_r)$.
Consider the action of 
$$
G=\symm_\alpha:=\symm_{\alpha_1} \times \cdots \times \symm_{\alpha_\ell},
$$
a {\it Young subgroup} of the symmetric group $\symm_r$, indexed by any composition 
$\alpha=(\alpha_1,\ldots,\alpha_\ell)$ of $r$.
In examining the $\symm_\alpha$-invariants
$(\Sym_{\pm} V)^{\symm_\alpha}=(\wedge V)^{\symm_\alpha}$,
and more generally, 
$$
(\Sym_{\pm} V^n)^{\symm_n[\symm_\alpha]}=(\wedge V^n)^{\symm_n[\symm_\alpha]},
$$
it was observed that their Hilbert series
seemed to depend only on the length $\ell$ of $\alpha$, not on
$\alpha$ itself.  

This can be explained as follows. When the full symmetric group $\symm_r$ acts on 
$\wedge V=\wedge(\theta_1,\ldots,\theta_r)$, 
one readily sees that the
$\symm_r$-invariants $(\wedge V)^{\symm_r}$ are $\kk$-spanned by the two elements
$1 \in \wedge^0 V=\kk$ and $\theta_1+\cdots+\theta_r \in \wedge^1 V=V$.
This lets one easily analyze the action of $\symm_\alpha$ acting on $\wedge V$.
The direct sum decomposition
$V=V^{\alpha_1} \oplus \cdots \oplus V^{\alpha_\ell}$ gives rise to a graded tensor product decomposition
$$
\wedge V \cong \wedge V^{\alpha_1} \otimes \cdots \otimes  \wedge V^{\alpha_\ell}
$$
compatible with the action of  $\symm_\alpha=\symm_{\alpha_1} \times \cdots \times \symm_{\alpha_\ell}$ on each tensor factor.
This gives an isomorphism
\begin{align*}
(\wedge V)^{\symm_\alpha} 
 &\cong 
 (\wedge V^{\alpha_1})^{\symm_{\alpha_1}} \otimes \cdots \otimes  (\wedge V^{\alpha_\ell})^{\symm_{\alpha_\ell}}
 \cong \wedge( \theta^{(1)}, \ldots, \theta^{(\ell)}),
 \end{align*}
 where $\theta^{(i)}$ is the sum of
 the $\alpha_i$ exterior variables in degree one of $\wedge V^{\alpha_i}$.
Consequently,
$$
\Hilb((\wedge V)^{\symm_\alpha},u)
=
\Hilb
\left( \wedge
(\theta^{(1)}, \ldots, \theta^{(\ell)}), u 
\right)=
\sum_{j=0}^\ell \binom{\ell}{j} u^j, 
$$
depending only upon $\ell$.  Using this,
Theorem~\ref{thm:compiled-wreath-hilbs} then compiles the
$\symm_n[\symm_\alpha]$-invariants Hilbert series:
\begin{equation}
\label{eq:Thibon-calculation}
\sum_{n=0}^\infty t^n 
\Hilb((\wedge V^n)^{\symm_n[\symm_\alpha]},u)
=
\displaystyle\prod_{\substack{1 \leq j \leq \ell\\j \text{ odd}}} (1+tu^j)^{ \binom{\ell}{j} }
\bigg/
\displaystyle\prod_{\substack{0 \leq j \leq \ell\\j \text{ even}}} (1-tu^j)^{ \binom{\ell}{j} }.
\end{equation}
The authors are grateful to J.-Y. Thibon
for deriving \eqref{eq:Thibon-calculation}
via plethystic calculation (personal communication, April 2024),
when asked if the left side had a nice generating function.  
\end{example}

\begin{example}
\label{ex:diagonally-symmetric-example}
The case where $G=\{1_V\}$ acts trivially on
$V=V_{\bar{0}} \oplus V_{\bar{1}}$
is already of some interest.
Given the $\kk$-bases $x_1,x_2,\ldots,x_{r_0}$
for $V_{\bar{0}}$
and $\theta_1,\theta_2,\ldots,\theta_{r_1}$
for $V_{\bar{1}}$,
one can index $\kk$-bases for $V^n=V_{\bar{0}}^n 
\oplus V_{\bar{1}}^n$
with $n \times r_0$ and $n \times r_1$ matrices $X_{n,r_0}, \Theta_{n,r_0}$ of variables, containing $r_0$ sets of $n$ commuting variables and  $r_1$ sets of $n$ anticommuting variables:
\begin{equation}
\label{eq:matrices-of-variables}
X_{n,r_0}=\left[\begin{matrix}
    x_{11} & \cdots & x_{1r_0}\\
    x_{21} & \cdots & x_{2r_0}\\
    \vdots & \vdots & \vdots \\
       x_{n1} & \cdots & x_{nr_0}\\
\end{matrix}
\right],
\quad 
\Theta_{n,r_0}=\left[\begin{matrix}
    \theta_{11} & \cdots & \theta_{1r_1}\\
    \theta_{21} & \cdots & \theta_{2r_1}\\
    \vdots & \vdots & \vdots \\
       \theta_{n1} & \cdots & \theta_{nr_1}\\
\end{matrix}
\right]
\end{equation}
Then one can regard 
$
\Sym_\pm V^n=\Sym_\pm(X_{n,r_0},\Theta_{n,r_1})
$
as the superpolynomial algebra  on these commuting and anticommuting variable sets.
Here $\symm_n(=\symm_n[G])$ acts on the matrix variable sets by permuting the row indices.  The $\symm_n$-invariant subalgebra $\Sym_\pm(X_{n,r_0},\Theta_{n,r_1})^{\symm_n}$ is often called 
the {\it ring of diagonally symmetric functions} in several variable sets of size $n$, with $r_0$ sets of {\it bosonic} variables and 
$r_1$ sets of {\it fermionic} variables; see  Bergeron \cite{Bergeron-bosonic-fermionic, Bergeron-bosonic}, Orellana and Zabrocki \cite[\S 4]{OrellanaZabrocki}.

Note that when $n=1$, one has for $\Sym_\pm V = 
\Sym_\pm(x_1,\ldots,x_{r_0},\theta_1,\ldots,\theta_{r_1})$ that
$$
\Hilb( \Sym_\pm V, q, u)
=\frac{(1+u)^{r_1}}{(1-q)^{r_0}}
=\sum_{\substack{i \geq 0,\\0 \leq j \leq r_1}} 
\underbrace{\multichoose{r_0}{i}
\binom{r_1}{j}}_{a_{ij}} q^i u^j,
$$
where $\multichoose{r_0}{i}:=\binom{r_0+i-1}{i}$ counts $i$-multisubsets of $\{1,2,\ldots,r_0\}$. Thus Theorem~\ref{thm:compiled-wreath-hilbs} implies here that
\begin{align}
\label{eq:diagonally-symmetric-hilb}
\sum_{n=0}^\infty t^n
\Hilb( \Sym_\pm(X_{n,r_0},\Theta_{n,r_1})^{\symm_n}, q, u)
& =\prod_{
\substack{i \geq 0,\\
1 \leq j \leq r_1\\j\text{ odd}}} 
(1+t q^i u^j)^{\multichoose{r_0}{i} \binom{r_1}{j}}
\bigg/
\prod_{
\substack{i \geq 0,\\
0 \leq j \leq r_1\\j\text{ even}}} 
(1-t q^i u^j)^{\multichoose{r_0}{i} \binom{r_1}{j}},\\
\label{eq:diagonally-antisymmetric-hilb}
\sum_{n=0}^\infty t^n
\Hilb( \Sym_\pm(X_{n,r_0},\Theta_{n,r_1})^{\symm_n,\sgn}, q, u)
& =\prod_{
\substack{i \geq 0,\\
0 \leq j \leq r_1\\j\text{ even}}} 
(1+t q^i u^j)^{\multichoose{r_0}{i} \binom{r_1}{j}}
\bigg/
\prod_{
\substack{i \geq 0,\\
1 \leq j \leq r_1\\j\text{ odd}}} 
(1-t q^i u^j)^{\multichoose{r_0}{i} \binom{r_1}{j}}.
\end{align}
It is worth noting that \eqref{eq:diagonally-symmetric-hilb} is consistent with a basis for $\Sym_\pm(X_{n,r_0},\Theta_{n,r_1})^{\symm_n}$ constructed by Orellana and Zabrocki in \cite[Cor.~13]{OrellanaZabrocki}; see also their excellent historical discussion.

Specializing to $r_1=0$,
one is considering the diagonal
action of $\symm_n$ on  
the commutative polynomial algebra $\Sym_\pm(V^n)=\Sym[X_{n,r_0}]$ in the $r_0$ sets of $n$ variables $X_{n,r_0}$ from \eqref{eq:matrices-of-variables}.  Here the
$\symm_n$-invariant subalgebra $\Sym[X_{n,r_0}]^{\symm_n}$ is sometimes called the ring of  {\it multisymmetric} or
{\it MacMahon symmetric polynomials};
see, e.g., Gessel \cite[Part II]{Gessel}, Olver and Shakiban \cite[\S2]{OlverShakiban}, Vaccarino \cite{Vaccarino}. 
The Hilbert series of $\Sym[X_{n,r_0}]^{\symm_n}$ relates to generating functions for {\it vector} or {\it multipartite partitions}, as discussed in Andrews \cite[Chap.~12]{Andrews}. 
One of motivations in the work of Solomon \cite{Solomon-partition-identities} was to use the connection with invariant theory in order to explain combinatorial results of Wright and Gordon on multipartite partitions; see also the discussion in Andrews \cite[pp. 210-211]{Andrews}.
\end{example}

\begin{example}
\label{ex:superspace}
Specializing Example~\ref{ex:diagonally-symmetric-example}
further to
$r_0=r_1=1$ is still interesting.  Here
one has 
\begin{align*}
\Sym_\pm V^n=\Sym_{\pm}(X_{n,r_0},\Theta_{n,r_1})
&=\Sym_\pm (x_1,\ldots,x_n,\theta_1,\ldots,\theta_n)=:\Sym_\pm(\xx,\theta)\\
&=\Sym(x_1,\ldots,x_n) \otimes \wedge(\theta_1,\ldots,\theta_n),
\end{align*}
sometimes called the {\it coordinate ring of superspace}, or the ring of {\it differential forms on $\kk^n$ with polynomial coefficients}.
Here
$\symm_n$ acts
diagonally permuting subscripts.
The diagonally symmetric functions
$\Sym_\pm(\xx,\theta)^{\symm_n}$
and its various $\kk$-bases have been studied as {\it symmetric functions in superspace} by Desrosiers, Lapointe, Mathieu and collaborators; see, e.g., \cite{DesrosiersLapointeMathieu}.

In this context, equations
\eqref{eq:diagonally-symmetric-hilb},
\eqref{eq:diagonally-antisymmetric-hilb} specialized to $r_0=r_1=1$ assert the following:
\begin{align}
\notag
\sum_{n=0}^\infty 
t^n 
\Hilb(\Sym_\pm(\xx,\theta)^{\symm_n},q,u)
&=\prod_{i=0}^\infty \frac{
(1+tuq^i)}
{
(1-tq^i)}\\
\label{eq:recovering-typeA-Solomon}
&=\sum_{n=0}^\infty t^n
\frac{(1+u)(1+qu)\cdots (1+q^{n-1}u)}
{(1-q)(1-q^2)\cdots (1-q^n)},\\
\notag
\sum_{n=0}^\infty 
t^n 
\Hilb(\Sym_\pm(\xx,\theta)^{\symm_n,\sgn},q,u)
&=\prod_{i=0}^\infty\frac{
(1+tq^i)}
{
(1-tuq^i)}\\
\label{eq:recovering-typeA-Shepler}
&=\sum_{n=0}^\infty t^n
\frac{(u+1)(u+q)\cdots (u+q^{n-1})}
{(1-q)(1-q^2)\cdots (1-q^n)}.
\end{align}
Here the rightmost equalities in both \eqref{eq:recovering-typeA-Solomon},
\eqref{eq:recovering-typeA-Shepler}
applied the {\it $q$-binomial theorem}\footnote{
The rightmost equality in \eqref{eq:recovering-typeA-Solomon} applied \eqref{eq:q-binomial-theorem} with $a=-u$; \eqref{eq:recovering-typeA-Shepler} applied \eqref{eq:q-binomial-theorem} with $a=-u^{-1}$
followed by $t \mapsto ty$.}
\cite[eqn. (20)]{GasperRahman}
\begin{equation}
\label{eq:q-binomial-theorem}
\frac{(at;q)_\infty}{(t;q)_\infty}=
\sum_{n=0}^\infty t^n \frac{(a;q)_n}{(q;q)_n},
\end{equation}
where $(a;q)_n:=(1-a)(1-qa)\cdots(1-q^{n-1} a)=\prod_{i=0}^{n-1} (1-aq^i)$ and $(a;q)_\infty:=\prod_{i=0}^\infty (1-aq^i)$.  

Meanwhile,
equating coefficients of $t^n$
on the far left and far right sides
in \eqref{eq:recovering-typeA-Solomon}, \eqref{eq:recovering-typeA-Shepler} predicts that 
$\Sym_\pm(\xx,\theta)$
should have these Hilbert series for its $\symm_n$-invariants and $\symm_n$-antiinvariants:
\begin{align}
\label{eq:typeA-Solomon}
\Hilb(\Sym_\pm(\xx,\theta)^{\symm_n},q,u)
&=
\prod_{i=1}^n \frac{1+q^{i-1}u}
{1-q^i},\\
\label{eq:typeA-Shepler}
\Hilb(\Sym_\pm(\xx,\theta)^{\symm_n,\sgn},q,u)
&=
\prod_{i=1}^n \frac{u+q^{i-1}}
{1-q^i}.
\end{align}
These are consistent with the following known results of Solomon and of Shepler 
on $\symm_n$
thought of as a {\it reflection group}.
Earlier work of Solomon \cite{Solomon-invariants} (see also Kane \cite[\S 22]{Kane}, Smith \cite[\S9.3]{Smith}) shows that 
$\Sym_\pm(\xx,\theta)^{\symm_n}$ may be regarded as an {\it exterior algebra} with
coefficients in the (commutative) ring
of {\it symmetric polynomials}
$$
\kk[\xx]^{\symm_n}=\kk[x_1,\ldots,x_n]^{\symm_n}=\kk[e_1(\xx),\ldots,e_n(\xx)].
$$
Here $e_i=e_i(\xx)$ is the $i^{th}$ {\it elementary symmetric function} in $x_1,\ldots,x_n$, with $\deg(e_i)=i$.  Hence
$\Hilb(\kk[\xx]^{\symm_n},q)=\prod_{i=1}^n \frac{1}{1-q^i}$.
The exterior generators for $\Sym_\pm(\xx,\theta)^{\symm_n}$
over $\kk[\xx]^{\symm_n}$ are
$\{ de_1,\ldots,de_n\}$,
where $de_i:=\sum_{j=1}^n \frac{\partial e_i}{\partial x_j} \theta_j$ 
lies in $\ZZ^2$-degree tracked by $q^{i-1}u^1$.
This implies \eqref{eq:typeA-Solomon}.
Reflection group results of Shepler \cite{Shepler-thesis, Shepler} construct a similar exterior algebra structure over $\kk[\xx]^{\symm_n}$ on $\Sym_\pm(\xx,\theta)^{\symm_n,\sgn}$, implying \eqref{eq:typeA-Shepler};
see the calculation in Reiner, Shepler and Sommers \cite[Thm.~2.11]{ReinerSheplerSommers}.
\end{example}

\section*{Acknowledgements}
The authors thank Jean-Yves Thibon for deriving \eqref{eq:Thibon-calculation}, which helped spur
this investigation.  They also thank Sasha Voronov for pointing them to the helpful reference \cite{Manin}.
Both authors supported by NSF grant DMS-2053288.

\section{\texorpdfstring{$\ZZ_2$}{Z2}-graded vector spaces and superalgebras}
\label{sec:superalgebras}
For background on some of this material, see Manin \cite[Chap.~3]{Manin}. Throughout, abbreviate $\ZZ_2:=\ZZ/2\ZZ=\{\bar{0},\bar{1}\}$, and
let $\kk$ be any field not of characteristic two. All direct sums and tensor products will be over the field $\kk$, that is, $\oplus:=\oplus_\kk$ and $\otimes:=\otimes_\kk$.

\begin{definition} \rm
    A {\it $\ZZ_2$-graded $\kk$-vector space} is a $\kk$-vector space $V$ along with a choice of direct sum decomposition $V=V_{\bar{0}} \oplus V_{\bar{1}}$. 
A {\it supercommutative algebra} is an associative  $\kk$-algebra $A$ which is also a $\ZZ_2$-graded $\kk$-vector space $A=A_{\bar{0}} \oplus A_{\bar{1}}$, whose multiplication satisfies these compatibilities:
    \begin{itemize} 
    \item It {\it respects the $\ZZ_2$-grading}, meaning $A_{\bar{i}}  A_{\bar{j}} \subset A_{\bar{i}+\bar{j}}$ 
    for all $\bar{i},\bar{j} \in \ZZ_2$, and
    \item It is {\it graded-commutative}, meaning  
    $a b = (-1)^{\bar{i} \cdot \bar{j}} b a$ if $a \in A_{\bar{i}}$ and $b \in  A_{\bar{j}}
    $ for $\bar{i},\bar{j} \in \ZZ_2$.
    \end{itemize}
\end{definition}

\begin{example}
Given any $\ZZ_2$-graded $\kk$-vector space $V=V_{\bar{0}} \oplus V_{\bar{1}}$, one can form the {\it free supercommutative algebra  (or superpolynomial algebra) generated by $V$} as
$$
\Sym_\pm V:=\Sym(V_{\bar{0}}) \otimes \wedge(V_{\bar{1}}).    
$$
Here $\Sym U, \wedge U$ denote
the usual {\it symmetric, exterior algebras} on $U$, considered as
certain quotient $\kk$-vector spaces of
the {\it tensor algebra} $T(U)=\bigoplus_{d=0}^\infty T^d(U)$,
where $T^d(U)=U^{\otimes d}=U \otimes \cdots \otimes U$.
Multiplication in $\Sym_\pm V$ is componentwise over the tensor product, that is,
$
(a \otimes b)(a' \otimes b')
=a a' \otimes b b'
$
for $a,a' \in \Sym(V_{\bar{0}})$
and $b,b' \in \wedge(V_{\bar{1}})$.
Alternatively, as in the Introduction, one can
define $\Sym_\pm V$ as a quotient of the tensor
algebra $T(V)$ as follows:
\begin{equation}
\label{eq:superpolynomial-quadratic-relations} 
\Sym_\pm V =T(V)/I
\end{equation}
with $I$ the two-sided ideal 
generated by all
$a\otimes b-(-1)^{\bar{i} \cdot \bar{j}} \, b \otimes a$
where $a \in V_{\bar{i}}, b \in V_{\bar{j}}$
with $\bar{i},\bar{j} \in \ZZ_2$.

The presentation of $\Sym_\pm V$ given in \eqref{eq:superpolynomial-quadratic-relations} explains the name {\it free supercommutative algebra}, since it shows that $\Sym_\pm V$ is characterized by a {\it universal property}:  given any supercommutative algebra $A=A_{\bar{0}} \oplus A_{\bar{1}}$ and $\kk$-linear maps $f_{\bar{i}}: V_{\bar{i}} \rightarrow A_{\bar{i}}$
for $\bar{i}=\bar{0},\bar{1}$, one
can extend $f_{\bar{0}}, f_{\bar{1}}$ uniquely to a $\kk$-algebra map $f: \Sym_\pm V \rightarrow A$.

The $\ZZ_2$-graded $\kk$-vector space structure on $\Sym_\pm V$ is
actually a specialization of its natural {\it $\ZZ^2$-grading} under the map
$
\ZZ^2 \rightarrow \ZZ_2
$
that sends $(i,j) \mapsto \bar{j}$:
\begin{align*}
\Sym_\pm V &= \bigoplus_{(i,j) \in \NN^2} (\Sym_\pm V)_{ij}, \text{ where }
(\Sym_\pm V)_{ij}:=
\Sym^i(V_{\bar{0}}) \otimes \wedge^j(V_{\bar{1}})\\
&=
\underbrace{
\bigoplus_{\substack{(i,j):\\j \text{ even}}} (\Sym_\pm V)_{ij}}_{(\Sym_\pm V)_{\bar{0}}}
\quad \oplus \quad 
 \underbrace{
 \bigoplus_{\substack{(i,j):\\j \text{ odd}}}
(\Sym_\pm V)_{ij}}_{(\Sym_\pm V)_{\bar{1}}}.
\end{align*}

\end{example}

\begin{example}
The presentation of $\Sym_\pm V$ in \eqref{eq:superpolynomial-quadratic-relations} also shows that if one regards it as a $\ZZ$-graded $\kk$-algebra generated by its degree one component $V$, then the relations in \eqref{eq:superpolynomial-quadratic-relations} are {\it quadratic}.  It turns out that $\Sym_\pm V$ is more strongly a {\it Koszul algebra} in the sense of Priddy \cite{Priddy};
see Fr\"oberg \cite{Froberg} and Polishchuk and Positselski \cite{PolishchukPositselski}. 
An important role in Section~\ref{sec:shuffles} is played by its
{\it Koszul dual quadratic algebra}, what Manin  \cite[\S3.2.6]{Manin}) calls the {\it superexterior algebra} $\Lambda_\kk(T)$ and which we shall write as $\wedge_\pm V$.  It has a very similar
(quadratic dual\footnote{To check that the two presentations \eqref{eq:superpolynomial-quadratic-relations}, \eqref{eq:superpolynomial-shriek-quadratic-relations} are quadratic duals, one must note that since $2 \in \kk^\times$, the presentation $\Sym_\pm U=T(U)/I$ from \eqref{eq:superpolynomial-quadratic-relations} has $a\otimes a=\frac{1}{2}(a\otimes a + a \otimes a) \in I$ for $a \in V_{\bar{1}}$, while the presentation $\wedge_\pm V=T(V)/J$ from \eqref{eq:superpolynomial-shriek-quadratic-relations} similarly has $a\otimes a \in J$ for $a \in V_{\bar{0}}$.}) presentation
\begin{equation}
\label{eq:superpolynomial-shriek-quadratic-relations}
\wedge_\pm V :=T(V)/J
\end{equation}
where $T(V)$ is the tensor algebra and $J$ is the two-sided ideal 
generated by
$a \otimes b+(-1)^{\bar{i} \cdot \bar{j}} \, b \otimes a$
for $a \in V_{\bar{i}}, b \in V_{\bar{j}}$
with $\bar{i},\bar{j} \in \ZZ_2$.
The Koszul dual relation between $\Sym_\pm V$ and $\wedge_\pm V$ is reflected in a relation between their ($\ZZ$-graded) Hilbert series \cite[Cor.~2.2]{PolishchukPositselski}:
$$
\Hilb(\Sym_\pm V,t)=\frac{(1+t)^{r_1}}{(1-t)^{r_0}} \text{ and }
\Hilb(\wedge_\pm V,t)=\frac{(1+t)^{r_0}}{(1-t)^{r_1}}
$$
which satisfy 
$
\Hilb(\wedge_\pm V,t) \cdot  \Hilb(\Sym_\pm V,-t)=1.
$
\end{example}

\section{Super Molien Theorem}
\label{sec:super-Molien}

As mentioned in the Introduction,
given a $\ZZ_2$-graded $\kk$-vector
space $V=V_{\bar{0}} \oplus V_{\bar{1}}$, a subgroup 
$G$ of $GL(V)$ {\it respects the $\ZZ_2$-grading} if both $V_{\bar{0}}, V_{\bar{1}}$ are $G$-stable,
or equivalently, $G$ is a subgroup of $GL(V_{\bar{0}}) \times GL(V_{\bar{1}})$.
Throughout this section, we will assume $G$ is a such a subgroup, and
write its elements as
$g=(g^{(0)},g^{(1)})$ where $g^{(0)}=g|_{V_{\bar{0}}}$ and $g^{(1)}=g|_{V_{\bar{1}}}$. 

In this setting, the $G$-invariant subalgebra
$
(\Sym_\pm V)^G
$
inherits from $\Sym_\pm V$
the structure of a supercommutative $\kk$-algebra, as well as its $\ZZ^2$-grading.  We wish to compute its
$\ZZ^2$-graded {\it Hilbert series},
and more generally, for any group homomorphism/character $\chi: G \rightarrow \kk^\times$, the
Hilbert series for the
$\kk$-subspace of
{\it $\chi$-relative invariants}
$(\Sym_\pm V)^{G,\chi}$ 
inside $\Sym_\pm V$.
The following {\it super Molien theorem} straightforwardly generalizes the commutative case known as Molien's Theorem \cite{Molien} (see also \cite{Benson,Kane,Smith,Stanley-invariants}), along with Solomon's computation \cite[\S4]{Solomon-invariants} when $V_{\bar{0}} \cong V_{\bar{1}}$ 
and $G$ is a finite group diagonally embedded 
in $GL(V)$; see also 
Kane \cite[\S 22-4]{Kane}, Smith \cite[\S 9.3]{Smith}.

\begin{prop}
\label{prop:super-molien}
If a finite group $G$ respects the $\ZZ_2$-grading on $V=V_{\bar{0}} \oplus V_{\bar{1}}$, and $|G| \in \kk^\times$, then
$$
\Hilb((\Sym_\pm V)^G,q,u)
=\frac{1}{|G|}
\sum_{g \in G}
\frac{\det(1_{V_{\bar{1}}}+u \cdot g^{(1)})}{\det(1_{V_{\bar{0}}}-q \cdot g^{(0)})}.
$$
More generally, for any homomorphism $\chi: G \rightarrow \kk^\times$, one has
$$
\Hilb((\Sym_\pm V)^{G,\chi},q,u)
=\frac{1}{|G|}
\sum_{g \in G}
\chi(g^{-1}) \frac{\det(1_{V_{\bar{1}}}+u \cdot g^{(1)})}{\det(1_{V_{\bar{0}}}-q \cdot g^{(0)})}.
$$
\end{prop}
\begin{proof}
Since $|G|$ lies in $\kk^\times$, for any finite-dimensional $\kk G$-module $U$,
one has that the idempotent $e^\chi_G:=\frac{1}{|G|} 
\sum_{g \in G} \chi(g^{-1}) \cdot g$ in $\kk G$ acts as a projection 
$e^{\chi}_G: U \twoheadrightarrow U^{G,\chi}$.  Consequently 
$$
\dim_\kk U^{G,\chi} = \trace(e^{\chi}_G: U \rightarrow U)=
\frac{1}{|G|} 
\sum_{g \in G} \chi(g^{-1}) \cdot  
\trace(g: U \rightarrow U).
$$
Applying this to each bidegree $(i,j)$ of $\Sym_\pm V$ and summing gives
\begin{equation}
\label{eq:bigraded-trace-calculation}
\Hilb(\Sym_\pm V^{G,\chi},q,u)
= 
\frac{1}{|G|} 
\sum_{g \in G}\chi(g^{-1})  
\sum_{(i,j) \in \NN^2} q^i u^j
\trace(g: (\Sym_\pm V)_{ij} 
\rightarrow (\Sym_\pm V)_{ij}).
\end{equation}
One can extend $\kk$ to a field $\overline{\kk}$ where $g^{(0)}, g^{(1)}$ have all of their eigenvalues
$(\lambda_1,\ldots,\lambda_{r_0}),
(\mu_1,\ldots,\mu_{r_1})$, so that  the action of  $g^{(0)},g^{(1)}$ on $V_{\bar{0}}, V_{\bar{1}}$ has a $\overline{\kk}$-basis that is triangularized. 
This also leads to bases of monomial tensors for each
$(\Sym_\pm V)_{ij}=\Sym^i(V_{\bar{0}}) 
\otimes \wedge^j(V_{\bar{1}})$
in which $g^{(0)},g^{(1)}$ are again triangularized.  This
gives a $\ZZ^2$-graded trace calculation, which combined with \eqref{eq:bigraded-trace-calculation} 
finishes the proof of the proposition: 
$$
\sum_{(i,j) \in \NN^2} q^i u^j
\trace(g|_{(\Sym_\pm V)_{ij}})
=\prod_{j=1}^{r_1} 
(1+\mu_j y) 
\cdot
\prod_{i=1}^{r_0} (1+\lambda_i q + \lambda_i^2 q^2 + \cdots)
=\frac{\det(1_{V_{\bar{1}}}+u \cdot g^{(1)})}{\det(1_{V_{\bar{0}}}-q \cdot g^{(0)})}. \qedhere
$$
\end{proof}

\begin{remark}
\label{rem:relative-invariants-superMolien}
When $\kk=\CC$, a general form
of Proposition~\ref{prop:super-molien} holds
with the same proof, for the {\it character} $\chi:G \rightarrow \CC$ of any irreducible $\CC G$-module $U$,
along the lines of Stanley \cite[Thm.~2.1]{Stanley-invariants}.
Define the $(\Sym_\pm V)^G$-module $(\Sym_\pm V)^{G,\chi}$ of {\it $\chi$-relative invariants} to be the sum of all $\chi$-isotypic components
within each graded component of $\Sym_\pm V$; 
this generalizes the definition of $\chi$-relative invariants from the Introduction.
Then the element
$e_G^\chi:=\frac{\chi(1)}{|G|} \sum_{g \in G}
\chi(g^{-1}) g$ in $\CC G$ is a (central) idempotent that again projects $U \twoheadrightarrow U^{G,\chi}$ onto the $\chi$-isotypic components in finite-dimensional $\CC G$-modules $U$ (see, e.g., Etingof \cite[Prob.~4.5.2]{Etingof}).  From this one 
similarly deduces
\begin{equation}
\label{eq:general-relative-invariants-Molien}
 \Hilb((\Sym_\pm V)^{G,\chi},q,u)
=\frac{\chi(1)}{|G|}
\sum_{g \in G}
\chi(g^{-1}) \frac{\det(1_{V_{\bar{1}}}+u \cdot g^{(1)})}{\det(1_{V_{\bar{0}}}-q \cdot g^{(0)})}.
\end{equation}
\end{remark}

Proposition~\ref{prop:super-molien}
has a particularly convenient rephrasing
when working with a permutation group $P$ acting either on $V=V_{\bar{0}}$ or $V=V_{\bar{1}}$,
so that $\Sym_\pm V$ is either the polynomial algebra $\Sym V$ or the exterior algebra $\wedge V$.
This rephrasing uses
the ring of {\it symmetric functions}
(see Macdonald \cite{Macdonald}, Stanley \cite[Chap.~7]{Stanley-EC2}) mentioned in the Introduction: it is
the polynomial ring
$$
\Lambda:=\Lambda_\QQ=\QQ[p_1,p_2,\ldots]
$$
in a countable variable set $p_1,p_2,\ldots$ called {\it power sums}.  Consequently $\Lambda$ has a $\QQ$-basis $\{p_{\lambda}\}$
with $p_\lambda:=p_{\lambda_1} p_{\lambda_2} \cdots p_{\lambda_\ell}$ indexed
by {\it partitions} $\lambda=(\lambda_1,\ldots,\lambda_\ell)$
with $\lambda_1 \geq \lambda_2 \geq \cdots \geq \lambda_\ell >0$.

This rephrasing also uses 
{\it P\'olya's cycle index} $Z_P$ 
of a permutation group $P$ (see, e.g., Stanley \cite[\S7.24]{Stanley-EC2}) and its signed
variant $Z^{\sgn}_P$, that appeared in the Introduction as \eqref{eq:cycle-index-defn}, \eqref{eq:signed-cycle-index-defn}: 
\begin{align*}
Z_P&:=\frac{1}{|P|}
\sum_{\sigma \in P} p_{\lambda(\sigma)}
=
\frac{1}{|P|}
\sum_{\sigma \in P} \prod_{\text{ cycles }C \text{ of }\sigma} p_{|C|},\\
Z^{\sgn}_P&:=
\frac{1}{|P|}
\sum_{\sigma \in P} \sgn(\sigma) \cdot p_{\lambda(\sigma)}
=
\frac{1}{|P|}
\sum_{\sigma \in P} \prod_{\text{ cycles }C \text{ of }\sigma}(-1)^{|C|-1} p_{|C|}
=\omega(Z_P).
\end{align*}
Here $\lambda(\sigma)$ is the cycle type partition of $\sigma$, and
$\sgn: \symm_n \rightarrow \{\pm 1\}$ is the {\it sign homomorphism}, while $\omega: \Lambda \rightarrow \Lambda$ is the {\it fundamental involution} on symmetric
functions $\Lambda=\QQ[p_1,p_2,\ldots]$ which can be defined by
$\omega(p_r):=(-1)^{r-1} p_r$ for $r=1,2,\ldots$.  Lastly, we need the 
the notion of {\it plethystic substitution}; see
Loehr and Remmel \cite{LoehrRemmel},  Macdonald \cite[\S I.8]{Macdonald} and Stanley \cite[Defn.~A.2.6]{Stanley-EC2}.

\begin{definition} \rm
\label{def:univariate-plethystic-subsitution}
For a univariate power series $f(q)=\sum_{i=0}^\infty a_i q^i$ in $\QQ[[q]]$, one has a ring homomorphism
$\Lambda \rightarrow \QQ[[q]]$, called the {\it plethystic substitution}
$g \mapsto g[f(q)]$, defined uniquely by sending
each $p_r \mapsto p_r[f(q)]:=f(q^r)=\sum_{i=0}^\infty a_i q^{ri}$.  Similarly, for a
bivariate series $f(q,u)$ in $\QQ[[q,u]]$, one defines the plethystic substitution $\Lambda \rightarrow \QQ[[q,u]]$, 
sending  
$p_r\mapsto p_r[f(q,u)] := f(q^r,u^r)$.
\end{definition}

\begin{cor}
\label{cor:Molien-for-permutation-groups}
For $P$ a subgroup of $\symm_n$
permuting coordinates in $V=\kk^n$, with
$|P|$ in $\kk^\times$, 
and $\sgn: P \rightarrow \{\pm 1\}$
the sign homomorphism, one has these
plethystic Hilbert series formulas:
\begin{align}
\label{eq:commutative-Molien-for-permutation-groups}
\Hilb((\Sym V)^P,q)&
=Z_P\left[1/(1-q)\right]
=Z_P\left[1+q+q^2+q^3+\cdots \right],\\
\label{eq:anti-commutative-Molien-for-permutation-groups}
\Hilb((\wedge V)^P,-u)
&=Z_P\left[1-u\right],\\
\label{eq:commutative-sign-Molien-for-permutation-groups}
\Hilb((\Sym V)^{P,\sgn},q)&
=Z^{\sgn}_P\left[1/(1-q)\right],\\
\label{eq:anti-commutative-sign-Molien-for-permutation-groups}
\Hilb((\wedge V)^{P,\sgn},-u)
&=Z^{\sgn}_P\left[1-u\right].
\end{align}
In the special case that
$r_0=r_1=n$ and one embeds $P$ diagonally into
$\symm_n \times \symm_n$, so that it acts diagonally
on $\Sym_\pm V=\Sym(\kk^n) \otimes \wedge(\kk^n)$, then one has
\begin{align}
\label{eq:diagonal-super-Molien-for-permutation-groups}
\Hilb((\Sym_\pm V)^P,q,-u)
=Z_P\left[\frac{1-u}{1-q}\right],\\
\label{eq:diagonal-sign-super-Molien-for-permutation-groups}
\Hilb((\Sym_\pm V)^{P,\sgn},q,-u)
=Z^{\sgn}_P\left[\frac{1-u}{1-q}\right].
\end{align}
\end{cor}
\noindent
Note 
\eqref{eq:diagonal-super-Molien-for-permutation-groups} 
implies \eqref{eq:commutative-Molien-for-permutation-groups}, 
\eqref{eq:anti-commutative-Molien-for-permutation-groups} setting $u=0, q=0$, respectively,
and 
\eqref{eq:diagonal-sign-super-Molien-for-permutation-groups} similarly
implies \eqref{eq:commutative-sign-Molien-for-permutation-groups}, 
\eqref{eq:anti-commutative-sign-Molien-for-permutation-groups}.
\begin{proof}
Compare Proposition~\ref{prop:super-molien} and definitions 
\eqref{eq:cycle-index-defn},
\eqref{eq:signed-cycle-index-defn} 
of $Z_P, Z^{\sgn}_P$ with these calculations for $\sigma$:
\begin{align*}
\frac{1}{\det(1_V - q \cdot \sigma)}
&= \prod_{\substack{\text{cycles}\\C \text{ of }\sigma}} 
\frac{1}{1-q^{|C|}}
=p_{\lambda(\sigma)}\left[1/(1-q)\right],\\
\det(1_V - u \cdot \sigma)
&= \prod_{\substack{\text{cycles}\\C \text{ of }\sigma}} 
(1-u^{|C|})
=p_{\lambda(\sigma)}\left[1-u\right]\\
\frac{\det(1_{V_{\bar{0}}} -u \cdot \sigma)}
{\det(1_{V_{\bar{1}}} - q \cdot \sigma)}
&= \prod_{\substack{\text{cycles}\\C \text{ of }\sigma}} 
\frac{1-u^{|C|}}{1-q^{|C|}}
=p_{\lambda(\sigma)}\left[\frac{1-u}{1-q}\right].\qedhere
\end{align*}
\end{proof}

\section{Wreath products and proof of Theorem~\ref{thm:wreath-Molien-theorem}}
\label{sec:wreath-Molien}

This section generalizes a calculation in invariant theory of wreath products acting on commutative polynomial algebras $\Sym V=\kk[x_1,\ldots,x_n]$,
due to Dias and Stewart \cite{DiasStewart},
related to Solomon \cite{Solomon-partition-identities}.
The result is generalized from polynomial algebras $\Sym V$ to
also allow exterior algebras $\wedge V$ and superpolynomial algebras $\Sym_\pm V$.

\begin{definition} \rm
    Recall that for a subgroup $G$ of $GL(V)$ and subgroup $P$ of $\symm_n$, one can create their {\it wreath product} $G \wr P = P[G]$ as a subgroup of $GL(V^n)$ acting on 
    $$
    V^n:=V^{\oplus n}=\underbrace{V \oplus \cdots \oplus V}_{n\text{ summands}}
    $$
    as follows.  Let $\sigma \in P$ and $\gg=(g_1,\ldots,g_n) \in G^n$ act on $\vvv=(v_1,\ldots,v_n) \in V^n$ via
    \begin{align*}
    \sigma(\vvv)&=(v_{\sigma^{-1}(1)},\ldots,v_{\sigma^{-1}(n)}),\\
    \gg(\vvv)&=(g_1 v_1,\ldots,g_n v_n).
    \end{align*}
    This embeds both $P, G^n$ as subgroups of $GL(V^n)$, and $P[G]$ is the subgroup of $GL(V^n)$ generated by their images. It is a semidirect product $P \ltimes  G^n$, in which a pair $(\sigma,\gg)$ acts on $\vvv \in V^n$ as
    $$
    (\sigma,\gg)(\vvv)=
    (g_{\sigma^{-1}(1)} v_{\sigma^{-1}(1)}, 
    \,\, \ldots, \,\, 
    g_{\sigma^{-1}(n)} v_{\sigma^{-1}(n)}).
    $$
Also define the {\it sign homomorphism} $\sgn: P[G] \rightarrow \{\pm 1\}$ via $\sgn(\sigma, \gg):=\sgn(\sigma)$, i.e., the composite
    $$
    P[G] \twoheadrightarrow P[G]/G^n \cong P \hookrightarrow \symm_n \overset{\sgn}{\longrightarrow} \{\pm 1\}.
    $$
\end{definition}

When $V=V_{\bar{0}} \oplus V_{\bar{1}}$ is a $\ZZ_2$-graded $\kk$-vector space, it gives rise to a $\ZZ_2$-grading
$V^n=V_{\bar{0}}^n \oplus V_{\bar{1}}^n$.  Furthermore,
when a subgroup $G$ of $GL(V)$ respects the
$\ZZ_2$-grading on $V$,
then for any permutation subgroup $P$ of $\symm_n$, the wreath product $P[G]$ inside $GL(V^n)$ will respect
this $\ZZ_2$-grading on $V^n$.
We therefore refer to the elements
of $P[G]$ as 
$(\sigma,\gg)=((\sigma,\gg)^{(0)}),(\sigma,\gg)^{(1)})$ in $GL(V_{\bar{0}}^n) \times GL(V_{\bar{}}^n)$.

The following block matrix lemma \cite[Lemma 6.6]{DiasStewart} plays a key role in the proof of Theorem~\ref{thm:wreath-Molien-theorem}.
\begin{lemma}
\label{lem:block-cycle-row-reduction}
    For any matrices $A_1,\ldots,A_m \in \kk^{r \times r}$, one has
    \begin{equation*}
    \det\left( I_{rm}-
    \left[
    \begin{smallmatrix} 
    0 & 0 & \cdots  & 0 & A_1\\
    A_2 & 0 & \cdots  & 0 & 0\\
    0 & A_3 & \cdots  & 0&  0\\
    \vdots & \vdots& \ddots & \vdots & \vdots\\
    0 & 0 & \cdots & A_m & 0\\
    \end{smallmatrix}
    \right]\right) 
    = \det(I_r - A_m A_{m-1} \cdots A_2 A_1).
    \end{equation*}
\end{lemma}
\begin{proof}[Proof of Lemma~\ref{lem:block-cycle-row-reduction}.]
Equate determinants on the sides of this identity $AB=C$, shown for $m=4$:
\small
\begin{align*}
&\begin{pmatrix}
     I_r &\rvline&   &\rvline&  &\rvline&   &      \\\hline
     A_2&\rvline &I_r&\rvline&  &\rvline& &    \\\hline
     A_3A_2&\rvline&A_3&\rvline&I_r&\rvline& & \\\hline
     A_4A_3A_2&\rvline&A_4A_3&\rvline&A_4&\rvline&I_r& \\
 \end{pmatrix}
 \begin{pmatrix}
     I_r &\rvline&   &\rvline&  &\rvline&   -A_1      \\\hline
     -A_2&\rvline &I_r&\rvline&  &\rvline& &    \\\hline
         &\rvline& -A_3&\rvline&I_r&\rvline& & \\\hline
         &\rvline& &\rvline& -A_4&\rvline&I_r& \\
 \end{pmatrix}
= 
 \begin{pmatrix}
     I_r &\rvline&   &\rvline&  &\rvline&   &\rvline&  -A_1    \\\hline
     &\rvline &I_r&\rvline&  &\rvline& &\rvline&   -A_2A_1 \\\hline
     &\rvline& &\rvline&I_r&\rvline& &\rvline& -A_3A_2A_1 \\\hline
     &\rvline& &\rvline& &\rvline&I_r&\rvline & I_r-A_4A_3A_2A_1\\
 \end{pmatrix}
\end{align*}
\normalsize
Note $\det A=1$, while $\det B, \det C$ are the left, right sides in the lemma.
\end{proof}

We can now restate and prove here
Theorem~\ref{thm:wreath-Molien-theorem}
from the Introduction,
on wreath products $P[G]$. Its commutative case
for $\Sym(V^n)^{P[G]}$ appears in Dias and Stewart
\cite[\S6.2, Thm.~6.4]{DiasStewart}, while its commutative case for $\Sym(V^n)^{\symm_n[G]}$ and
$\Sym(V^n)^{\symm_n[G],\sgn}$ are implicit in Solomon
\cite[\S 2]{Solomon-partition-identities}.

\vskip.1in
\noindent
{\bf Theorem~\ref{thm:wreath-Molien-theorem}.}
{\it 
In the above setting, if $|G|,|P|$ lie in $k^\times$, then
\begin{align*}
\Hilb((\Sym_\pm V^n)^{P[G]},q,-u)
&=Z_P\left[\, \Hilb((\Sym_\pm V)^G,q,-u)\, \right]\\
\Hilb((\Sym_\pm V^n)^{P[G],\sgn},q,-u)
&=Z^{\sgn}_P\left[\, \Hilb((\Sym_\pm V)^G,q,-u)\, \right].
\end{align*}
}
\vskip.1in
\begin{proof}
Noting that $|P[G]|=|P| \cdot |G|^n$,
applying Proposition~\ref{prop:super-molien} to $P[G]$ yields these formulas:
\begin{align*}
\Hilb((\Sym_\pm V^n)^{P[G]},q,-u)
&=\frac{1}{|P[G]|}
\sum_{(\sigma,\gg)  \in P[G]}
\frac{\det(1_{V_{\bar{1}}}-u \cdot (\sigma,\gg)^{(1)})}
{\det(1_{V_{\bar{0}}}-q \cdot (\sigma,\gg)^{(0)})} \\
&=\frac{1}{|P|}
\sum_{\sigma \in P}
\left( \frac{1}{|G|^n}
\sum_{\gg  \in G^n}
\frac{\det(1_{V_{\bar{1}}}-u \cdot (\sigma,\gg)^{(1)})}
{\det(1_{V_{\bar{0}}}-q \cdot (\sigma,\gg)^{(0)})} \right),\\
\Hilb((\Sym_\pm V^n)^{P[G],\sgn},q,-u)
&=\frac{1}{|P|}
\sum_{\sigma \in P} \sgn(\sigma)
\left( \frac{1}{|G|^n}
\sum_{\gg  \in G^n}
\frac{\det(1_{V_{\bar{1}}}-u \cdot (\sigma,\gg)^{(1)})}
{\det(1_{V_{\bar{0}}}-q \cdot (\sigma,\gg)^{(0)})} \right).
\end{align*}
Comparing these with the definitions
\eqref{eq:cycle-index-defn}, \eqref{eq:signed-cycle-index-defn} of $Z_P, Z^{\sgn}_P$ as
$$
Z_P=\frac{1}{|P|}
\sum_{\sigma \in P} p_{\lambda(\sigma)}
\quad \text{ and } \quad
Z^{\sgn}_P=\frac{1}{|P|}
\sum_{\sigma \in P} \sgn(\sigma) \cdot  p_{\lambda(\sigma)}
$$
it suffices to prove the following identity for each permutation $\sigma$ in $\symm_n$:
\begin{equation}
\label{eq:wreath-permutation-traces}
\frac{1}{|G|^n}
\sum_{\gg  \in G^n}
\frac{\det(1_{V_{\bar{1}}^n}-u \cdot (\sigma,\gg)^{(1)})}
{\det(1_{V_{\bar{0}}^n}-q \cdot (\sigma,\gg)^{(0)})}
=p_{\lambda(\sigma)}\left[ \Hilb((\Sym_\pm V)^G,q,-u)\right].
\end{equation}
Writing $\sigma = \sigma_1 \cdots \sigma_\ell$ as a product of cycles $\sigma_i$ of size $\lambda_i$, the right side of \eqref{eq:wreath-permutation-traces} is a product:
\begin{align}
\notag
p_{\lambda(\sigma)}\left[ \Hilb((\Sym_\pm V)^G,q,-u)\right] 
&=\prod_{i=1}^\ell
p_{\lambda_i}\left[ \Hilb((\Sym_\pm V)^G,q,-u)\right]\\
\label{eq:product-form-for-right}
&=\prod_{i=1}^\ell
 \Hilb((\Sym_\pm V)^G,q^{\lambda_i},-u^{\lambda_i})
 \end{align}
Meanwhile, decomposing 
$V^n = \bigoplus_{i=1}^\ell V^{\lambda_i}$ according to the cycles $\sigma_i$ of $\sigma$, with a corresponding decomposition
$\gg=(\gg(1),\gg(2),\ldots,\gg(\ell))$ where
$\gg(i) \in G^{\lambda_i}$, one has this
determinant
factorization:
\begin{equation}
\label{eq:charpoly-factorizations}
\frac{\det(1_{V_{\bar{1}}^n}-u \cdot (\sigma,\gg)^{(1)})}
{\det(1_{V_{\bar{0}}^n}-q \cdot (\sigma,\gg)^{(0)})}
= \prod_{i=1}^\ell
\frac{\det(1_{V_{\bar{1}}^{\lambda_i}}-u \cdot (\sigma,\gg(i))^{(1)})}
{\det(1_{V_{\bar{0}}^{\lambda_i}}-q \cdot (\sigma,\gg(i))^{(0)})}.
\end{equation}
Using \eqref{eq:charpoly-factorizations} and interchanging a sum and product, one can rewrite the left side of 
\eqref{eq:wreath-permutation-traces} as a product:
\begin{equation}
\label{eq:interchanged-product-for-sum}
\frac{1}{|G|^n}
\sum_{\gg  \in G^n}
\frac{\det(1_{V_{\bar{1}}^n}-u \cdot (\sigma,\gg)^{(1)})}
{\det(1_{V_{\bar{0}}^n}-q \cdot (\sigma,\gg)^{(0)})}
= \prod_{i=1}^\ell
\frac{1}{|G|^{\lambda_i}}
\sum_{\gg  \in G^{\lambda_i}}
\frac{\det(1_{V_{\bar{1}}^{\lambda_i}}-u \cdot (\sigma,\gg(i))^{(1)})}
{\det(1_{V_{\bar{0}}^{\lambda_i}}-q \cdot (\sigma,\gg(i))^{(0)})}.
\end{equation}
Comparing the right sides in \eqref{eq:product-form-for-right}, \eqref{eq:interchanged-product-for-sum}, in order to prove the desired equation \eqref{eq:wreath-permutation-traces}, it suffices to prove the following
special case for any $m$-cycle $\sigma$ in $\symm_m$:
\begin{equation}
\label{eq:single-cycle-identity}
\frac{1}{|G|^m}
\sum_{\gg  \in G^m}
\frac{\det(1_{V_{\bar{1}}^{m}}-u \cdot (\sigma,\gg)^{(1)})}
{\det(1_{V_{\bar{0}}^{m}}-q \cdot (\sigma,\gg)^{(0)})}
= \Hilb((\Sym_\pm V)^G,q^m,-u^m).
\end{equation}
This follows from a calculation whose steps are explained below:
\begin{align*}
\label{eq:single-cycle-identity}
\frac{1}{|G|^m}
\sum_{\gg  \in G^m}
\frac{\det(1_{V_{\bar{1}}^{m}}-u \cdot (\sigma,\gg)^{(1)})}
{\det(1_{V_{\bar{0}}^{m}}-q \cdot (\sigma,\gg)^{(0)})}
&\overset{(a)}{=}
\frac{1}{|G|^m}
\sum_{\gg  \in G^m}
\frac{\det\left(1_{V_{\bar{1}}}-
(u \cdot g_1^{(1)}) \cdots  (u \cdot g_m^{(1)}) 
\right)}
{\det\left(1_{V_{\bar{0}}}-
(q \cdot g_1^{(0)}) \cdots  (q \cdot g_m^{(0)}) 
\right)}\\
&\overset{(b)}{=}\frac{1}{|G|^m}
\sum_{\gg  \in G^m}
\frac{\det\left(1_{V_{\bar{1}}}-
u^m (g_1 \cdots g_{m})^{(1)}\right)
}
{\det\left(1_{V_{\bar{0}}}-
q^m (g_1 \cdots g_{m})^{(0)}\right)
}\\
&=
\frac{1}{|G|^m}
\sum_{g
\in G}
\sum_{\substack{\gg  \in G^m:\\ g_1 \cdots g_m=g}}
\frac{\det(1_{V_{\bar{1}}}-
u^m g^{(1)})
}
{\det(1_{V_{\bar{0}}}-
q^m g^{(0)})
}\\
&\overset{(c)}{=}
\frac{1}{|G|}
\sum_{g
\in G}
\frac{\det(1_{V_{\bar{1}}}-
u^m g^{(1)})
}
{\det(1_{V_{\bar{0}}}-
q^m g^{(0)})
}
\overset{(d)}{=} \Hilb((\Sym_\pm V)^G,q^m,-u^m).
\end{align*}
Equality (a) applied Lemma~\ref{lem:block-cycle-row-reduction} twice for each summand $\gg \in G^m$: 
\begin{itemize}
    \item once in the denominator with $(A_1,\ldots,A_m)=(q \cdot g^{(0)}_1,\ldots,q \cdot g^{(0)}_m)$,
    \item once in the numerator with $(A_1,\ldots,A_m)=
(u \cdot g^{(1)}_1,\ldots,u \cdot g^{(1)}_m)$.
\end{itemize}
Equality (b) used the fact that $g,h \in G \subset GL(V_{\bar{0}}) \times GL(V_{\bar{1}})$
have $(gh)^{(i)}=g^{(i)} h^{(i)}$ for $i=0,1$.
Equality (c) arises since for each $g$ in $G$, there are $|G|^{m-1}$ solutions $\gg=(g_1,\ldots,g_m) \in G^m$ to the equation
$g_1 \cdots g_m=g$. Equality (d) is Proposition~\ref{prop:super-molien}, substituting $q \mapsto q^m$ and $u \mapsto -u^m$.
\end{proof}

\begin{remark}
 Dias and Stewart \cite[Thm.~6.4]{DiasStewart} work with $\kk=\RR$ or $\CC$ and in the commutative setting of $\Sym(V)$, but allow more general subgroups $G$ of $GL(V)$ which are {\it compact}, and not necessarily finite. Their methods replace all averages of the form
$\frac{1}{|G|} \sum_{g \in G} (-)$ and
$\frac{1}{|G|^n} \sum_{\gg \in G^n}(-)$ with
integrations $\int_G(-)dg$ and $\int_{G^n} (-)dg_1 \cdots dg_n$, using the {\it Haar measures} on $G, G^n$.  The same replacements prove a version of Theorem~\ref{thm:wreath-Molien-theorem} when $\kk=\RR$ or $\CC$, allowing $G$ to be
a compact group, but keeping $P$ as a finite
permutation subgroup of $\symm_n$.
\end{remark}

\begin{remark}
As in Remark~\ref{rem:relative-invariants-superMolien}, when $\kk=\CC$ 
one can generalize Theorem~\ref{thm:wreath-Molien-theorem} to a statement
on relative invariants.
Given $\chi: P \rightarrow \CC$ any complex irreducible character of a permutation subgroup $P$ of $\symm_n$, define its {\it $\chi$-cycle index polynomial} in $\Lambda$ to be
$$
Z^\chi_P:=\frac{\chi(1)}{|G|} \sum_{\sigma \in P} \chi(\sigma^{-1}) p_{\lambda(\sigma)}.
$$
One can also view $\chi$ as
an irreducible character of the wreath product $P[G]$, via inflation through the quotient surjection $P[G] \twoheadrightarrow P[G]/G^n \cong P$, that is, $\chi((\sigma,\gg))=\chi(\sigma)$.
Then using \eqref{eq:general-relative-invariants-Molien}, the proof of Theorem~\ref{thm:wreath-Molien-theorem} similarly shows that the 
$\chi$-relative invariants $(\Sym_\pm)^{P[G],\chi}$
have this Hilbert series: 
\[
\Hilb((\Sym_\pm V^n)^{P[G],\chi},q,-u)
=Z^{\chi}_P\left[\, \Hilb((\Sym_\pm V)^G,q,-u)\, \right].
\]
Then \eqref{eq:commutative-Molien-for-permutation-groups}, \eqref{eq:anti-commutative-Molien-for-permutation-groups},
\eqref{eq:commutative-sign-Molien-for-permutation-groups}, \eqref{eq:anti-commutative-sign-Molien-for-permutation-groups}, \eqref{eq:diagonal-super-Molien-for-permutation-groups}, \eqref{eq:diagonal-sign-super-Molien-for-permutation-groups}
also generalize in this context, replacing $Z_P, Z^{\sgn}_P$ with $Z^\chi_P$.

An interesting case arises when $P=\symm_n$ itself, so that the irreducible
characters $\chi=\chi^\lambda$ are indexed by partitions $\lambda$ of $n$.  Here $Z_{\symm_n}^{\chi^\lambda}=f^\lambda \cdot s_\lambda$, where $s_\lambda$ is the {\it Schur function} indexed by $\lambda$,
and $f^\lambda=\chi^{\lambda}(1)$ is the number of {\it standard Young tableaux} of shape $\lambda$; see Stanley \cite[\S7.18, (7.86)]{Stanley-EC2}.
\end{remark}

\begin{remark}
 There is a quicker path to 
 the Hilbert series expression \eqref{eq:wreath-product-Molien} for the $P[G]$-invariants in Theorem~\ref{thm:wreath-Molien-theorem}
 using Corollary~\ref{cor:Molien-for-permutation-groups}, whenever one is in the purely commutative or purely anticommutative case, and the subgroup $G$ of $GL(V)$ also happens to be a {\it permutation} subgroup of $\symm_r$.  One can
 then apply a famous result of P\'olya \cite[\S27,(1.40)]{PolyaRead}, asserting that the wreath product $P[G]$ for two permutation groups $P,G$ has cycle index
$Z_{P[G]}=Z_P[Z_G]$,
 where here the binary operation $(f,g) \mapsto f[g]$ denotes {\it plethystic composition} of symmetric functions $\Lambda \times \Lambda \rightarrow \Lambda$; see
Loehr and Remmel \cite{LoehrRemmel},  Macdonald \cite[\S I.8]{Macdonald} and Stanley \cite[Defn.~A.2.6]{Stanley-EC2}.  One can then deduce \eqref{eq:wreath-product-Molien}, e.g., in the commutative case,  by a calculation explained below:
\begin{align*}
\Hilb(\Sym(V^n)^{P[G]},q)
\overset{(a)}{=}Z_{P[G]}\left[\frac{1}{1-q}\right] 
&\overset{(b)}{=} Z_P\left[Z_G\right]
\left[\frac{1}{1-q} \right] \\
&\overset{(c)}{=} Z_P\left[Z_G\left[
\frac{1}{1-q}\right] \right] 
\overset{(d)}{=}Z_P\left[ \, \Hilb((\Sym_\pm V)^G,q)\, \right].
\end{align*}
Equality (a) applied \eqref{eq:commutative-Molien-for-permutation-groups} for $P[G]$.  
Equality (b) used 
 $Z_{P[G]}=Z_P[Z_G]$.  
Equality
(c) used the {\it associative compatibility} $
(f[g])\left[h(q)\right]=f\left[ g\left[h(q)\right] \right]$ of plethystic composition $f[g]$ for
 $f,g \in \Lambda$ with the plethystic substitutions $g\mapsto g[h(q)]$ for $h(q) \in \QQ[[q]]$.
Equality (d) applied \eqref{eq:commutative-Molien-for-permutation-groups} for $G$. 
 
\end{remark}

\section{Proof of Theorem~\ref{thm:compiled-wreath-hilbs}}
\label{sec:compiled-hilb-section}

Recall the statement of the theorem
from the Introduction.
\vskip.1in
\noindent
{\bf Theorem \ref{thm:compiled-wreath-hilbs}.}
{\it 
Let $V$ be a $\ZZ_2$-graded $\kk$-vector space, and $G$ a finite subgroup $G$ of $GL(V)$ that respects the grading.  When $\chr{\kk}=0$, one has these generating functions:
\begin{align}
\sum_{n=0}^\infty
t^n 
\Hilb((\Sym_\pm V^n)^{\symm_n[G]},q,u)
&=
\displaystyle
\prod_{\substack{(i,j) \in \NN^2:\\j \text{ odd}}}
(1+t q^{i} u^{j} )^{a_{ij}}
\bigg/
\displaystyle
\prod_{\substack{(i,j) \in \NN^2:\\j \text{ even}}}
(1-t q^{i} u^{j} )^{a_{ij}}, \tag{6}\\
\sum_{n=0}^\infty
t^n 
\Hilb((\Sym_\pm V^n)^{\symm_n[G],\sgn},q,u)
&=
\displaystyle
\prod_{\substack{(i,j) \in \NN^2:\\j \text{ even}}}
(1+t q^{i} u^{j} )^{a_{ij}}
\bigg/
\displaystyle
\prod_{\substack{(i,j) \in \NN^2:\\j \text{ odd}}}
(1-t q^{i} u^{j} )^{a_{ij}}.\tag{7}
\end{align}
}
\begin{proof}
Note that the full symmetric group $\symm_n$ has cycle index 
$Z_{\symm_n}=\frac{1}{n!}\sum_{\sigma \in \symm_n} 
p_{\lambda(\sigma)}=:h_n$,
and signed cycle index
$
Z^{\sgn}_{\symm_n}=\frac{1}{n!}\sum_{\sigma \in \symm_n} \sgn(w) \cdot p_{\lambda(\sigma)}=:e_n,
$
with these generating functions in $\Lambda[[t]]$ (see Macdonald \cite[\S I.2,p.~25]{Macdonald}, Stanley \cite[(7.22)]{Stanley-EC2}):
\begin{align*}
\symmfnH(t)&:=\sum_{n=0}^\infty t^n Z_{\symm_n} = \sum_{n=0}^\infty t^n h_n =
\exp\left( \sum_{r=1}^\infty \frac{p_r}{r}t^r \right),\\
\symmfnE(t)&:=\sum_{n=0}^\infty t^n Z^{\sgn}_{\symm_n} = \sum_{n=0}^\infty t^n e_n =
\exp\left( - \sum_{r=1}^\infty \frac{p_r}{r}(-t)^r \right)
=\frac{1}{\symmfnH(-t)}.\\
\end{align*}
Then Theorem~\ref{thm:wreath-Molien-theorem} gives the first step in this plethystic derivation of \eqref{eq:hilb-of-shuffle-algebra}:
\begin{align*}
&\sum_{n=0}^\infty
t^n 
\Hilb((\Sym_\pm V^n)^{\symm_n[G]},q,-u)
=\sum_{n=0}^\infty
t^n Z_{\symm_n}\left[ \Hilb(\Sym(V)^G,q,-u) \right]\\
&=\symmfnH(t)\left[ 
\sum_{(i,i) \in \NN^2}
a_{ij} q^i (-u)^j
\right]
=\exp\left( \sum_{r=1}^\infty \frac{p_r}{r}t^r \right)
\left[
\sum_{\substack{(i,j) \in \NN^2:\\j \text{ even}}}a_{ij} q^{i} u^{j}
-\sum_{\substack{(i,j) \in \NN^2:\\j \text{ odd}}} a_{ij} q^{i} u^{j}
\right]\\
&=
\prod_{\substack{(i,j) \in \NN^2:\\j \text{ even}}} 
\exp\left( \sum_{r=1}^\infty \frac{q^{ri} u^{rj}}{r}t^r \right)^{a_{ij}} 
\cdot 
\prod_{\substack{(i,j) \in \NN^2:\\j \text{ odd}}} 
\exp\left( \sum_{r=1}^\infty \frac{q^{ri} u^{rj}}{r}t^r \right)^{-a_{ij}}\\
&=
\displaystyle
\prod_{\substack{(i,j) \in \NN^2:\\j \text{ odd}}}
(1-t q^{i} u^{j} )^{a_{ij}}
\bigg/
\displaystyle
\prod_{\substack{(i,j) \in \NN^2:\\j \text{ even}}}
(1-t q^{i} u^{j} )^{a_{ij}}
\end{align*}
where the last step used this calculation: 
$$
\exp\left( \sum_{r=1}^\infty \frac{q^{ri} u^{rj}}{r}t^r \right)
= \exp\left( \sum_{r=1}^\infty \frac{(t q^{i} u^{j})^r}{r} \right)
= \exp\left(-\log(1-tq^iu^j)\right)
=\frac{1}{1-tq^i u^j}.
$$
Replacing $u$ by $-u$ then completes the proof of \eqref{eq:hilb-of-shuffle-algebra}.
The derivation for \eqref{eq:hilb-of-signed-shuffle-algebra} is analogous, essentially applying the involution $\omega: \Lambda \rightarrow \Lambda$, prior to the plethystic substitution.
\end{proof}

\begin{example}
The commutative case of Theorem~\ref{thm:compiled-wreath-hilbs}
asserts that, given $U=\kk^r$ with $\chr(\kk)=0$, and any finite subgroup $G$ of $GL(U)$ having $\Hilb((\Sym U)^G,q)=\sum_{i=0}^\infty a_i q^i$, one will have
\begin{align}
\label{eq:Solomon-invariants-specialization}
\sum_{n=0}^\infty 
t^n 
\Hilb(\Sym(U^n)^{\symm_n[G]},q)
&=\prod_{i=0}^\infty (1-t q^i )^{-a_i},\\
\label{eq:Solomon-anti-invariants-specialization}
\sum_{n=0}^\infty 
t^n 
\Hilb(\Sym(U^n)^{\symm_n[G],\sgn},q)
&=\prod_{i=0}^\infty (1+t q^i )^{a_i}.
\end{align}
These \eqref{eq:Solomon-invariants-specialization},\eqref{eq:Solomon-anti-invariants-specialization} are special cases of Solomon's \cite[Thm.~3.15, (3.17),(3.18)]{Solomon-invariants}, where his $m=1$, his $W$ is the trivial $G$-representation,
his graded $G$-vector space $V$ is our $\Sym(U)$, and his $x$ is our $q$.
\end{example}

\section{Shuffles, signed shuffles and
proof of Theorem~\ref{thm:shuffle-algebra-is-superpolynomial}}
\label{sec:shuffles}

We recall the statement of Theorem~\ref{thm:shuffle-algebra-is-superpolynomial} from the Introduction,
which involved these direct sums:
\begin{align*}
R_G:=\bigoplus_{n=0}^\infty
(R_G)_n &\text{ where }(R_G)_n:=
(\Sym_\pm V^n)^{\symm_n[G]},\\
R^{\sgn}_G:=\bigoplus_{n=0}^\infty (R^{\sgn}_G)_n &\text{ where }
(R^{\sgn}_G)_n :=(\Sym_\pm V^n)^{\symm_n[G],\sgn}.
\end{align*}
It also involved the superpolynomial algebra
$\Sym_\pm U$ and superexterior algebra $\wedge_\pm U$ defined for
any $\ZZ_2$-graded $\kk$-vector space $U$, as quotients of the tensor algebra $T(U)$
in \eqref{eq:superpolynomial-quadratic-relations} and \eqref{eq:superpolynomial-shriek-quadratic-relations}.
In particular, we will apply these constructions to this $\ZZ_2$-graded
$\kk$-vector space $U$ defined in
\eqref{eq:ensuing-Z2-grading}:
\begin{equation}
U:=(\Sym_\pm V)^G
= \underbrace{
\bigoplus_{\substack{(i,j) \in \NN^2:\\j\text{ even}}}
(\Sym_\pm V)_{ij}^G}_{U_{\bar{0}}
} 
\quad \oplus \quad
\underbrace{\bigoplus_{\substack{(i,j) \in \NN^2:\\j\text{ odd}}}
(\Sym_\pm V)_{ij}^G}_{U_{\bar{1}}
}.\tag{8} 
\end{equation}

\vskip.1in
\noindent
{\bf Theorem~\ref{thm:shuffle-algebra-is-superpolynomial}.}
{\it 
Under the hypotheses of Theorem~\ref{thm:compiled-wreath-hilbs},
one can define on $R_G, R^{\sgn}_G$
graded products,
\begin{itemize}
\item a shuffle product $\shuffle:(R_G)_a \otimes (R_G)_b \rightarrow (R_G)_{a+b}$, and 
\item a signed shuffle product  $\shuffle^{\sgn}: (R^{\sgn}_G)_a \otimes (R^{\sgn}_G)_b \rightarrow (R^{\sgn}_G)_{a+b}$,
\end{itemize}
for which one has $\ZZ^3$-graded algebra isomorphisms
\begin{align*}
\Sym_\pm \left( (\Sym_\pm V)^G \right)
& \cong R_G \text{ with its  }\shuffle\text{ product},\\
\wedge_\pm\left( (\Sym_\pm V)^G \right)
& \cong R^{\sgn}_G \text{ with its  }\shuffle^{\sgn}\text{ product}.
\end{align*}
}
\vskip.1in

Even the following (commutative) invariant theory corollary is new, as far as we know.
\begin{cor}
In the $V=V_{\bar{0}}$ case of Theorem~\ref{thm:shuffle-algebra-is-superpolynomial}, the shuffle product
$\shuffle$
on 
$$
R_G=\bigoplus_{n=0}^\infty \Sym(V^n)^{\symm_n[G]}
$$
makes it isomorphic to the polynomial algebra $\Sym((\Sym V)^G)$ as a bigraded commutative algebra.
\end{cor}

To define $\shuffle, \shuffle^{\sgn}$,
we start with the special case
where $G=\{1_V\}$, and consider
\begin{align*}  
R:&=R_{\{1_V\}} 
=\bigoplus_{n=0}^\infty R_n \text{ where }
R_n:=(\Sym_\pm V^n)^{\symm_n},\\   
R^{\sgn}&:=R^{\sgn}_{\{1_V\}} 
=\bigoplus_{n=0}^\infty
R^{\sgn}_n \text{ where }
R^{\sgn}_n:=
(\Sym_\pm V^n)^{\symm_n,\sgn}.
\end{align*}

It will be helpful to recall some definitions and facts on Young subgroups,
as defined in Example~\ref{ex:motivating-Thibon-example}, along with their {\it minimum-length coset representatives}, and their relation to
shuffles.

\begin{definition} \rm
For each {\it composition} $\alpha=(\alpha_1,\ldots,\alpha_\ell)$ with $\alpha_i \in \{1,2,\ldots\}$ and
$\sum_{i=1}^\ell \alpha_i=n$, consider the {\it Young subgroup} $\symm_\alpha \cong
\symm_{\alpha_1} \times \cdots \times \symm_{\alpha_\ell}$ inside the symmetric group $\symm_n$, consisting of permutations that independently permute
the values within the blocks of this set  partition:
\begin{equation}
\label{eq:composition-set-partition}
\begin{array}{rl}
\{1,2,\ldots,n\}
&=\{1,2,\ldots,\alpha_1\} \\
&\quad \sqcup \{\alpha_1+1, \alpha_1+2,\ldots,
\alpha_1+\alpha_2\}\\
&\quad \sqcup \{\alpha_1+\alpha_2+1, \alpha_1+\alpha_2+2,\ldots,
\alpha_1+\alpha_2+\alpha_3\}\\
& \qquad \vdots \\
&\quad \sqcup \{ n-\alpha_\ell+1,n-\alpha_\ell+2,\ldots,n \}.
\end{array}
\end{equation}
\end{definition}

We will regard 
$W=\symm_n$ as a {\it Coxeter group} 
with {\it Coxeter system} $(W,S)$ in which $S=\{s_1,\ldots,s_{n-1}\}$ for $s_i=(i,i+1)$.
Each $\sigma$ in $\symm_n$ has {\it (Coxeter group) length} 
\begin{align*}
\ell(\sigma)&:=\min\{\ell: \sigma=s_{i_1} \cdots s_{i_\ell} \text{ for some }s_{i_j} \in S\}\\
&=|\{(i,j): 1\leq i<j\leq n \text{ and }\sigma(i) > \sigma(j)\}|.
\end{align*}
This length also determines
the {\it sign} of a permutation: $\sgn(\sigma)=(-1)^{\ell(w)}$. One can then view the Young subgroup $\symm_\alpha$ as a {\it standard parabolic subgroup}, with a set $\symm^\alpha$ of distinguished
{\it minimum-length coset representatives}
for the cosets $\symm_\alpha \sigma$ in $\symm_\alpha \backslash \symm_n$;
see
Bj\"orner-Brenti \cite[\S 2.4]{BjornerBrenti}, Humphreys [\S 1.10]\cite{Humphreys}.
This gives each $\sigma \in \symm_n$ a unique 
{\it length-additive decomposition} 
$$
\sigma = \sigma_\alpha \cdot \sigma^{\alpha} 
\text{ where }   \sigma_\alpha \in \symm_\alpha, \sigma^{\alpha} \in \symm^\alpha,
$$
and $\ell(\sigma)= \ell(\sigma_\alpha)+\ell(\sigma^\alpha)$.
One can view the elements of $\symm^{\alpha}$ as {\it shuffles} of the
blocks in the partition \eqref{eq:composition-set-partition},
where the entries of each block appear left-to-right. 

\begin{definition} \rm
The {\it unsigned} and {\it signed 
$\alpha$-shuffles} are these elements of
the group algebra $\kk \symm_n$:
$$
\sh_\alpha:=\sum_{\sigma \in \symm^\alpha} \sigma 
\quad \text{ and } \quad
\sh^{\sgn}_\alpha:=\sum_{\sigma \in \symm^\alpha} \sgn(\sigma) \cdot \sigma.\\
$$
\end{definition}

\begin{example}
\label{ex:shuffle-example}
For $\alpha=(2,2)$, one has this $(2,2)$-shuffles,and signed $(2,2)$-shuffle inside $\kk \symm_4$:
\begin{align*}
\sh_{2,2}
&= 
+\left(
\begin{matrix}
1234\\12{\color{darkred}34}
\end{matrix}
\right) 
+
\left(
\begin{matrix}
1234\\1{\color{darkred}3}2{\color{darkred}4}
\end{matrix}
\right)
+
\left(
\begin{matrix}
1234\\1{\color{darkred}34}2
\end{matrix}
\right)
+
\left(
\begin{matrix}
1234\\{\color{darkred}3}12{\color{darkred}4}
\end{matrix}
\right)
+
\left(
\begin{matrix}
1234\\{\color{darkred}3}1{\color{darkred}4}2
\end{matrix}
\right) 
+\left(
\begin{matrix}
1234\\{\color{darkred}34}12
\end{matrix}
\right),\\
\sh^{\sgn}_{2,2}
&= 
+\left(
\begin{matrix}
1234\\12{\color{darkred}34}
\end{matrix}
\right) 
{\color{darkred}-}
\left(
\begin{matrix}
1234\\1{\color{darkred}3}2{\color{darkred}4}
\end{matrix}
\right)
+
\left(
\begin{matrix}
1234\\1{\color{darkred}34}2
\end{matrix}
\right)
+
\left(
\begin{matrix}
1234\\{\color{darkred}3}12{\color{darkred}4}
\end{matrix}
\right)
{\color{darkred}-}
\left(
\begin{matrix}
1234\\{\color{darkred}3}1{\color{darkred}4}2
\end{matrix}
\right) 
+\left(
\begin{matrix}
1234\\{\color{darkred}34}12
\end{matrix}
\right) 
\end{align*}
\end{example}

One can then define for compositions $\alpha$ of $n$ these maps $\Sym_\pm V^{\alpha_1} \otimes 
\cdots \otimes
\Sym_\pm V^{\alpha_\ell}
\rightarrow
\Sym_\pm V^n$:
\begin{itemize}
\item 
A map $\mu_\alpha$ which is this composition
$$
\mu_\alpha: \Sym_\pm V^{\alpha_1} \otimes 
\cdots \otimes
\Sym_\pm V^{\alpha_\ell}
\overset{\iota}{\hookrightarrow} 
\Sym_\pm V^n \otimes
\cdots \otimes \Sym_\pm V^n
\overset{\mu}{\rightarrow}
\Sym_\pm V^n
$$
where $\iota$ comes from inclusion $V^{\alpha_i} \hookrightarrow V^n=\oplus_{i=1}^\ell V^{\alpha_i}$, and $\mu$ is multiplication in $\Sym_\pm V^n$.
\item
  Two maps $\shuffle_\alpha, \shuffle^{\sgn}_\alpha$ 
  which follow $\mu_\alpha$ with either
  $\sh_\alpha$ or $\sh^{\sgn}_\alpha$:
\begin{align*}
\shuffle_\alpha&:=\sh_\alpha \circ \mu_\alpha\\
\shuffle^{\sgn}_\alpha&:=\sh^{\sgn}_\alpha \circ \mu_\alpha.\\
\end{align*}
\end{itemize}

\begin{example}
Take $r_0=r_1=1$ as in Example~\ref{ex:superspace}, so that $V=\spn_\kk\{x, \theta\}$ and 
\begin{align*}
\Sym_\pm V&=\Sym_\pm(x,\theta)=\kk[x] \otimes \wedge(\theta),\\
(\Sym_\pm V^n)
&=\Sym_\pm(x_1,\ldots,x_n,\theta_1,\ldots,\theta_n)
=\kk[x_1,\ldots,x_n] \otimes \wedge(\theta_1,\ldots,\theta_n)
\end{align*}
with the diagonal action of $\symm_n$ permuting subscripts:  $\sigma(x_i)=x_{\sigma^{-1}(i)}$ and  
$\sigma(\theta_i)=\theta_{\sigma^{-1}(i)}$.
Here is an example of the map
$\shuffle^{\sgn}_{2,2}:
\Sym_\pm(V^2) \otimes \Sym_\pm(V^2)
 \longrightarrow
\Sym_\pm(V^4),
$
using $\sh^{\sgn}_{2,2}$ from
Example~\ref{ex:shuffle-example}:
\begin{align*}
x_1^2 x_2 \theta_2 \quad \shuffle^{\sgn}_{2,2} \quad 
{\color{darkred}x_1^5 x_2^7 \theta_1 \theta_2}
\,\,\, = \,\,\, \sh^{\sgn}_{2,2}\left(
x_1^2 x_2 \theta_2 \cdot 
{\color{darkred}x_3^5 x_4^7 \theta_3 \theta_4}
\right)
&=x_1^2 x_2 \theta_2 \cdot 
{\color{darkred}x_3^5 x_4^7 \theta_3 \theta_4}\\
&\quad - x_1^2 x_3 \theta_3 \cdot 
{\color{darkred}x_2^5 x_4^7 \theta_2 \theta_4}\\
&\quad + x_1^2 x_4 \theta_4 \cdot 
{\color{darkred}x_2^5 x_3^7 \theta_2 \theta_3}\\
&\quad + x_2^2 x_3 \theta_3 \cdot 
{\color{darkred}x_1^5 x_4^7 \theta_1 \theta_4}\\
&\quad - x_2^2 x_4 \theta_4 \cdot 
{\color{darkred}x_1^5 x_3^7 \theta_1 \theta_3}\\
&\quad + x_3^2 x_4 \theta_4 \cdot 
{\color{darkred}x_1^5 x_2^7 \theta_1 \theta_2}.
\end{align*}
The (unsigned) shuffle product
$x_1^2 x_2 \theta_2 \,\, \shuffle_{2,2} \,\,
x_1^5 x_2^7 \theta_1 \theta_2$
has the same six summands with plus signs.
\end{example}

\begin{prop}
\label{prop:shuffle-product-well-defined}
For any $\ZZ_2$-graded vector space $V=V_{\bar{0}} \oplus V_{\bar{1}}$,
these composite maps
\begin{align}
\label{eq:shuffle-product-as-composite}
\shuffle_{a,b}&:=\sh_{a,b} \circ \mu_{a,b}:
\Sym_\pm V^a \otimes \Sym_\pm V^b
\longrightarrow \Sym_\pm V^{a+b},\\
\label{eq:signed-shuffle-product-as-composite}
\shuffle^{\sgn}_{a,b}&:=\sh^{\sgn}_{a,b} \circ \mu_{a,b}:
\Sym_\pm V^a \otimes \Sym_\pm V^b
\longrightarrow \Sym_\pm V^{a+b}.
\end{align}
restrict to well-defined maps
\begin{align*}
\shuffle_{a,b}:& R_a \otimes R_{b}
\longrightarrow
R_{a+b},\\
\shuffle^{\sgn}_{a,b}:& R^{\sgn}_a \otimes R^{\sgn}_{b}
\longrightarrow
R^{\sgn}_{a+b}
\end{align*}
making both
$R, R^{\sgn}$
into associative, $\ZZ^3$-graded $\kk$-algebras.
\end{prop}

Before proving this, let us illustrate it with an example.

\begin{example}
Take $r_0=3,r_1=2$, and  to avoid column indices, 
name the $\kk$-bases 
of $V_{\bar{0}}$ by $\{x,y,z\}$ and of $V_{\bar{1}}$ by $\{\alpha,\beta\}$.  Then
\begin{align*}
R_1&=\Sym_\pm V=\Sym_\pm(x,y,z,\alpha,\beta),\\
R_n&=
\Sym_\pm(V^n)^{\symm_n}\\
&=\Sym_\pm(x_1,\ldots,x_n, \,\,
y_1,\ldots,y_n,\,\,
z_1,\ldots,z_n,\,\,
\alpha_1,\ldots,\alpha_n,\,\,
\beta_1,\ldots,\beta_n)^{\symm_n}.
\end{align*}
Here is a particular example of 
$\shuffle_{1,2}: R_1 \otimes R_2 \rightarrow R_3$:
\begin{align*}
(x_1^5 y_1^5 z_1^3 \alpha_1) \shuffle
\color{darkred}{(z_1 \beta_2 + z_2 \beta_1)} 
&=\sum_{ \sigma \in \{1{\color{darkred}23}, \,\, {\color{darkred}2}1{\color{darkred}{3}}, \,\, {\color{darkred}23}1\}}
\sigma \left( (x_1^5 y_1^5 z_1^3 \alpha_1) \cdot 
\color{darkred}{(z_2 \beta_3 + z_3 \beta_2)}
\right)\\
&=\sum_{ \sigma \in \{1{\color{darkred}23}, \,\, {\color{darkred}2}1{\color{darkred}{3}}, \,\, {\color{darkred}23}1\}}
\sigma \left( x_1^5 y_1^5 z_1^3 {\color{darkred}{z_2}} \alpha_1 {\color{darkred}{\beta_3}}
+x_1^5 y_1^5 z_1^3 {\color{darkred}{z_3}} \alpha_1 {\color{darkred}{\beta_2}} \right)\\
&=x_1^5 y_1^5 z_1^3 {\color{darkred}{z_2}} \alpha_1 {\color{darkred}{\beta_3}}
+x_1^5 y_1^5 z_1^3 {\color{darkred}{z_3}} \alpha_1 {\color{darkred}{\beta_2}}\\
&+x_2^5 y_2^5 z_2^3 {\color{darkred}{z_1}} \alpha_2 {\color{darkred}{\beta_3}}
+x_2^5 y_2^5 z_2^3 {\color{darkred}{z_3}} \alpha_2 {\color{darkred}{\beta_1}}\\
&+x_3^5 y_3^5 z_3^3 {\color{darkred}{z_1}} \alpha_3 {\color{darkred}{\beta_2}}
+x_3^5 y_3^5 z_3^3 {\color{darkred}{z_2}} \alpha_3 {\color{darkred}{\beta_1}}.
\end{align*}
\end{example}

\begin{proof}[Proof of Proposition~\ref{prop:shuffle-product-well-defined}.]
The $\kk$-linearity of
$\shuffle_{a,b}, \shuffle^{\sgn}_{a,b}$
follows because $\mu_{a,b}$ is $\kk$-linear, and both $\sh_{a,b}$ and $\sh^{\sgn}_{a,b}$ apply a (signed) sum
of $\kk$-algebra automorphisms $\sigma$, which all act $\kk$-linearly.

Associativity of $\shuffle, \shuffle^{\sgn}$ arises 
due to these equalities:
\begin{align*}
\shuffle_{a+b,c} \circ (\shuffle_{a,b} \otimes 1_{R_c})
&= \shuffle_{a,b,c}
= \shuffle_{a,b+c} \circ (1_{R_a} \otimes \shuffle_{b,c}),\\
\shuffle^{\sgn}_{a+b,c} \circ (\shuffle^{\sgn}_{a,b} \otimes 1_{R_c})
&= \shuffle^{\sgn}_{a,b,c}
= \shuffle^{\sgn}_{a,b+c} \circ (1_{R_a} \otimes \shuffle^{\sgn}_{b,c}).
\end{align*}
For example, the key point in checking the equality 
$\shuffle^{\sgn}_{a+b,c} \circ (\shuffle^{\sgn}_{a,b} \otimes 1_{R_c})= \shuffle^{\sgn}_{a,b,c}$, is as follows. 
The composition $(a,b,c)$ refines the
composition $(a+b,c)$, giving a tower of parabolic/Young subgroups $$
\symm_{a,b,c} < \symm_{a+b,c} < \symm_{a+b+c},
$$
each of which is a Coxeter group in its own
right, with compatible Coxeter generators. This gives for each
minimum-length coset representative $\sigma \in \symm^{a,b,c}$ a 
unique factorization $\sigma=\sigma' \cdot \sigma''$, where $\sigma'$ lies $\symm^{a+b,c}$ and
$\sigma''$ is a minimum-length coset representatives for
$\symm_{a,b,c}\backslash \symm^{a+b,c}$;  compare with the proof of Bj\"orner-Brenti \cite[Cor.~2.4.6]{BjornerBrenti}.

The fact that the maps $\shuffle_{a,b}$ and $\shuffle^{\sgn}_{a,b}$ restrict from $\Sym_\pm V^a \otimes  \Sym_\pm V^b \rightarrow \Sym V^{a+b}$ to well-defined maps 
$R_a \otimes R_{b} \rightarrow R_{a+b}$
and $R^{\sgn}_a \otimes R^{\sgn}_{b} \rightarrow R^{\sgn}_{a+b}$ is argued as follows, say for $R^{\sgn}_{a+b}$. For any $\symm_a$-antisymmetric element
$A \in R^{\sgn}_a$ and $\symm_b$-antisymmetric element
$B \in R^{\sgn_b}$, the product $\mu_{a,b}(A \otimes B)$ will be
antisymmetric under the action of all $\sigma$ in 
the Young subgroup $\symm_{a,b}=\symm_a \times \symm_b$. But then $\shuffle_{a,b}(A \otimes B) = \sh_{a,b} (\mu_{a,b}(A \otimes B))$
will be $\symm_{a+b}$-antisymmetric, due to the definition of $\sh_{a,b}$ as a signed sum over the coset representatives
$\symm^{a,b}$ for $\symm_{a,b} \backslash \symm_{a+b}$. 
\end{proof}

\begin{prop}
\label{prop:degree-one-generation}
For any $\ZZ_2$-graded $\kk$-vector space $V=V_{\bar{0}} \oplus V_{\bar{1}}$ with $\chr(\kk)=0$,
both $\ZZ$-graded algebras $R=\bigoplus_{n=0}^\infty R_n$ and $R^{\sgn}=\bigoplus_{n=0}^\infty R^{\sgn}_n$, are generated in degree one, by $\Sym_\pm V$.
\end{prop}
\begin{proof}
Abbreviating the composition $1^n=(1,1,\ldots,1)$,
one must show surjectivity of these maps:
\begin{align}
\label{eq:unsigned-iso-in-degree-n}
\shuffle_{1^n}:
(\Sym_\pm V)^{\otimes n}
&=R_1^{\otimes n} \longrightarrow R_n =
(\Sym_\pm V^n)^{\symm_n},\\
\label{eq:signed-iso-in-degree-n}
\shuffle^{\sgn}_{1^n}:(\Sym_\pm V)^{\otimes n}
&=(R^{\sgn}_1)^{\otimes n} \longrightarrow R^{\sgn}_n =
(\Sym_\pm V^n)^{\symm_n,\sgn}.
\end{align}
We will exhibit $\kk$-spanning sets for the right sides, and
show they are in the image of these maps.  As in Example~\ref{ex:diagonally-symmetric-example}, after picking 
bases $x_1,\ldots,x_{r_0}$ for $V_{\bar{0}}$
and $\theta_1,\ldots,\theta_{r_1}$ for
$V_{\bar{1}}$, one can identify
\begin{align*}
\Sym_\pm V&=\Sym_\pm(x_1,\ldots,x_{r_0},
\theta_1,\ldots,\theta_{r_1}),\\
\Sym_\pm V^n&=\Sym_{\pm}(X_{n,r_0},\Theta_{n,r_1}),\\
R_n&=\Sym_\pm(X_{n,r_0},\Theta_{n,r_1})^{\symm_n},\\
R^{\sgn}_n&=\Sym_\pm(X_{n,r_0},\Theta_{n,r_1})^{\symm_n,\sgn}
\end{align*}
where $X_{n,r_0},\Theta_{n,r_1}$
are the matrices of commuting and anticommuting variables in \eqref{eq:matrices-of-variables}.
A typical $\kk$-basis element for
$\Sym_{\pm}(X_{n,r_0},\Theta_{n,r_1})$ is a monomial
indexed by matrices $A \in \NN^{n \times r_0}, B \in \{0,1\}^{n \times r_1}$:
$$
X_{n,r_0}^A \Theta_{n,r_1}^B:=
\prod_{\substack{i=1,\ldots,n\\j=1,\ldots,r_0}}  x_{ij}^{a_{ij}} 
\cdot 
\prod_{\substack{i=1,\ldots,n\\j=1,\ldots,r_1}}  \theta_{ij}^{b_{ij}}.
$$
Since $\chr(\kk)=0$, the elements
$e=\sum_{\sigma \in \symm_n} \sigma$
and $e^{\sgn}=\sum_{\sigma \in \symm_n} \sgn(\sigma) \cdot \sigma$ in $\kk\symm_n$
will project onto $\symm_n$-invariant 
and $\symm_n$-antiinvariant subspaces.
Therefore 
\begin{itemize}
\item $R_n=\Sym_\pm(X_{n,r_0},\Theta_{n,r_1})^{\symm_n}$ is spanned by 
$\sum_{\sigma \in \symm_n} \sigma(X_{n,r_0}^A \Theta_{n,r_1}^B)
$, and
\item $R^{\sgn}_n=\Sym_\pm(X_{n,r_0},\Theta_{n,r_1})^{\symm_n,\sgn}$ is spanned by
$\sum_{\sigma \in \symm_n} \sgn(\sigma) \cdot \sigma(X_{n,r_0}^A \Theta_{n,r_1}^B)$.
\end{itemize}
However, 
these are images of elements under
the maps \eqref{eq:unsigned-iso-in-degree-n}, \eqref{eq:signed-iso-in-degree-n}, illustrated here for the latter:
\begin{align*}
\sum_{\sigma \in \symm_n} \sgn(\sigma) \cdot \sigma(X_{n,r_0}^A \Theta_{n,r_1}^B)
&=\left(
x_1^{a_{11}} x_2^{a_{12}} \cdots 
x_{r_0}^{a_{1r_0}} \cdot
\theta_1^{b_{11}} \theta_2^{b_{12}} \cdots 
\theta_{r_0}^{b_{1r_0}}
\right)\\
&\shuffle^{\sgn}
\left(
x_1^{a_{21}} x_2^{a_{22}} \cdots 
x_{r_0}^{a_{2r_0}} \cdot
\theta_1^{b_{21}} \theta_2^{b_{22}} \cdots 
\theta_{r_0}^{b_{2r_0}}
\right)\\
&\quad \vdots \\
&\shuffle^{\sgn}
\left(
x_1^{a_{n1}} x_2^{a_{n2}} \cdots 
x_{r_0}^{a_{nr_0}} \cdot
\theta_1^{b_{n1}} \theta_2^{b_{n2}} \cdots 
\theta_{r_0}^{b_{nr_0}}
\right). 
\end{align*}
See Example~\ref{ex:surjectivity-example} below.
Hence the maps \eqref{eq:unsigned-iso-in-degree-n},\eqref{eq:signed-iso-in-degree-n} surject,
completing the proof.
\end{proof}

\begin{example}
\label{ex:surjectivity-example}
As in
Example~\ref{ex:diagonally-symmetric-example}, let $(r_0,r_1)=(3,2)$ and
$\Sym_\pm V=\Sym_\pm(x,y,z,\alpha,\beta)$. Within
$$
R_4=\Sym_\pm(x_1,x_2,x_3,x_4, \,\,
y_1,y_2,y_3,y_4, \,\,
z_1,z_2,z_3,z_4, \,\,
\alpha_1,\alpha_2,\alpha_3,\alpha_4, \,\, \\
\beta_1,\beta_2,\beta_3,\beta_4)^{\symm_4}
$$
one finds this $\symm_4$-symmetrized monomial   
$
\sum_{\sigma \in \symm_4}
\sigma\left(X_{43}^A \Theta_{42}^B \right)
$
where 
\begin{align*}
    X_{43}^A \Theta_{42}^B
    \quad =\quad (x_1^5 y_1^3)
    \quad \cdot \quad (x_2^3 y_2 z_2^{10} \,\, \beta_2) 
    \quad \cdot \quad (x_3^5 z_3 \,\, \alpha_3  \beta_3)
    \quad \cdot \quad (x_4^5 z_4\,\, \alpha_4 \beta_4)
\end{align*}
Its $\symm_4$-symmetrization lies in the image of the map \eqref{eq:unsigned-iso-in-degree-n} for $n=4$, as this $\shuffle$ product:
$$
\sum_{\sigma \in \symm_4}
    \sigma\left(X_{43}^A \Theta_{42}^B \right)
   \quad =\quad x^5 y^3
    \quad \shuffle \quad x^3  y z^{10} \, \beta
     \quad \shuffle \quad x^5 z \, \alpha \beta
     \quad \shuffle \quad x^5 z \, \alpha \beta.
$$
\end{example}

\begin{prop}
\label{prop:quadratic-relations}
For $V$ a $\ZZ_2$-graded $\kk$-vector space $V=V_{\bar{0}} \oplus V_{\bar{1}}$, consider
the $\ZZ_2$-grading on
$$
R_1=R^{\sgn}_1=\Sym_\pm V
= \underbrace{
\bigoplus_{\substack{(i,j) \in \NN^2:\\j\text{ even}}}
(\Sym_\pm V)_{ij}}_{(\Sym_\pm V)_{\bar{0}}} 
\quad \oplus \quad
\underbrace{\bigoplus_{\substack{(i,j) \in \NN^2:\\j\text{ odd}}}
(\Sym_\pm V)_{ij}}_{(\Sym_\pm V)_{\bar{1}}}. 
$$
\begin{itemize}
\item[(i)] Under $\shuffle$, elements of $R_1$ satisfy the quadratic relations
from \eqref{eq:superpolynomial-quadratic-relations}:
elements of $(\Sym_\pm V)_{\bar{0}}$
commute with all of $R_1$,
and elements of $(\Sym_\pm V)_{\bar{1}}$
anticommute among themselves.
\item[(ii)]
Under $\shuffle^{\sgn}$, elements of $R^{\sgn}_1$ satisfy the quadratic relations
from \eqref{eq:superpolynomial-shriek-quadratic-relations}:
elements of $(\Sym_\pm V)_{\bar{0}}$
anticommute with all of $R^{\sgn}_1$,
and elements of $(\Sym_\pm V)_{\bar{1}}$
commute among themselves.
\end{itemize}
\end{prop}
\begin{proof}
With notations as in Example~\ref{ex:diagonally-symmetric-example}, a typical $\kk$-basis element
of $\Sym_\pm V$ is 
$$
\xx^\aa \theta^\bb
=x_1^{a_1} \cdots x_{r_0}^{a_{r_0}}
\cdot \theta_1^{b_1} \cdots \theta_{r_1}^{b_{r_1}}
\text{ 
where }
\aa\in \NN^{r_0}, \bb\in \{0,1\}^{r_1},
$$
and it lies
in $(\Sym_\pm V)_{\bar{0}}$ or $(\Sym_\pm V)_{\bar{1}}$
depending on whether
$|\bb|:=\sum_i b_i \equiv 0$ or $1 \bmod{2}$.
Using the abbreviations
\begin{align*}
\xx^\aa_{1*}&:=x_{11}^{a_1} \cdots x_{1r_0}^{a_{r_0}}, \qquad
\xx^\aa_{2*}:=x_{21}^{a_1} \cdots x_{2r_0}^{a_{r_0}},\\
\theta^\bb_{1*}&:=\theta_{11}^{b_1} \cdots \theta_{1r_0}^{b_{r_1}},\qquad
\theta^\bb_{2*}:=\theta_{21}^{b_1} \cdots \theta_{2r_0}^{b_{r_1}},
\end{align*}
one can then compute that in $R$ with $\shuffle$ product, one has
\begin{align*}
\xx^\aa \theta^\bb
\shuffle \xx^\cc \theta^\dd
&=\xx_{1*}^{\aa} \xx_{2*}^\cc
\cdot \theta_{1*}^\bb \theta_{2*}^\dd
+
\xx_{2*}^{\aa} \xx_{1*}^\cc
\cdot \theta_{2*}^\bb \theta_{1*}^\dd\\
\text{versus }\quad\xx^\cc\theta^\dd
\shuffle \xx^\aa \theta^\bb
&=
\xx_{1*}^{\cc} \xx_{2*}^\aa
\cdot \theta_{1*}^\dd \theta_{2*}^\bb
+
\xx_{2*}^{\cc} \xx_{1*}^\aa
\cdot \theta_{2*}^\dd \theta_{1*}^\bb\\
&=
(-1)^{|\bb|\cdot|\dd|}
\left( 
\xx_{2*}^{\aa} \xx_{1*}^\cc
\cdot \theta_{2*}^\bb \theta_{1*}^\dd
+
\xx_{1*}^{\aa} \xx_{2*}^\cc
\cdot \theta_{1*}^\bb \theta_{2*}^\dd
\right)\\
&=
(-1)^{|\bb|\cdot|\dd|}
\cdot
\xx^\aa \theta^\bb
\shuffle \xx^\cc \theta^\dd.
\end{align*}
Similarly, in $R^{\sgn}$
 with $\shuffle^{\sgn}$, one has
\begin{align*}
\xx^\aa \theta^\bb
\shuffle^{\sgn} \xx^\cc \theta^\dd
&=\xx_{1*}^{\aa} \xx_{2*}^\cc
\cdot \theta_{1*}^\bb \theta_{2*}^\dd
-
\xx_{2*}^{\aa} \xx_{1*}^\cc
\cdot \theta_{2*}^\bb \theta_{1*}^\dd\\
\text{versus }\quad\xx^\cc\theta^\dd
\shuffle^{\sgn} \xx^\aa \theta^\bb
&=
\xx_{1*}^{\cc} \xx_{2*}^\aa
\cdot \theta_{1*}^\dd \theta_{2*}^\bb
-
\xx_{2*}^{\cc} \xx_{1*}^\aa
\cdot \theta_{2*}^\dd \theta_{1*}^\bb\\
&=
(-1)^{|\bb|\cdot|\dd|}
\left( 
\xx_{2*}^{\aa} \xx_{1*}^\cc
\cdot \theta_{2*}^\bb \theta_{1*}^\dd
-
\xx_{1*}^{\aa} \xx_{2*}^\cc
\cdot \theta_{1*}^\bb \theta_{2*}^\dd
\right)\\
&=
(-1)^{1+|\bb|\cdot|\dd|}
\cdot
\xx^\aa \theta^\bb
\shuffle^{\sgn} \xx^\cc \theta^\dd.
\end{align*}
The signs in these commutations conform to \eqref{eq:superpolynomial-quadratic-relations}
for $\shuffle$ on $R$, and \eqref{eq:superpolynomial-shriek-quadratic-relations} for $\shuffle^{\sgn}$ on $R^{\sgn}$.
\end{proof}

\begin{proof}[Proof of Theorem~\ref{thm:shuffle-algebra-is-superpolynomial}.]
We wish to exhibit algebra isomorphisms
\begin{align}
\label{eq:brief-shuffle-algebra-isomorphism}
\Sym_\pm( (\Sym_\pm V)^G )& \cong R_G,\\
\label{eq:brief-signed-shuffle-algebra-isomorphism}
\wedge_\pm( (\Sym_\pm V)^G )& \cong R^{\sgn}_G.
\end{align}
As noted in the Introduction,
Theorem~\ref{thm:compiled-wreath-hilbs} shows the rings on either side in 
\eqref{eq:brief-shuffle-algebra-isomorphism}, \eqref{eq:brief-signed-shuffle-algebra-isomorphism}
have the same Hilbert series. 
Thus it suffices to exhibit
a graded ring injection or surjection
between them.

For $G=\{1_V\}$, a surjective map follows from
Propositions~\ref{prop:shuffle-product-well-defined}, \ref{prop:degree-one-generation}, \ref{prop:quadratic-relations}:
the rings $R, R^{\sgn}$ are associative,
generated by the degree one component $R_1=R^{\sgn}_1=\Sym_\pm V$, and satisfy the quadratic
relations from \eqref{eq:superpolynomial-quadratic-relations}, \eqref{eq:superpolynomial-shriek-quadratic-relations}.
This gives graded ring surjections, as desired, and
hence isomorphisms.

For the case of an arbitrary group $G$, one only needs to show
that the products $\shuffle, \shuffle^{\sgn}$ on $R, R^{\sgn}$
restrict to products on $R_G, R^{\sgn}_G$:  if so, then
the ring isomorphisms exhibited in the case $G=1_V$ will restrict to
injective maps, which must then be isomorphisms.

Given $A, B$ in $(\Sym_\pm V^a)^{\symm_a[G]}$ and $(\Sym_\pm V^b)^{\symm_b[G]}$
(respectively, in $(\Sym_\pm V^a)^{\symm_a[G],\sgn}$ and $(\Sym_\pm V^b)^{\symm_b[G],\sgn}$),
one needs to check
that $A \shuffle B$ lies in
$(\Sym_\pm V)^{\symm_{a+b}[G]}$
(respectively, $A \shuffle^{\sgn} B$ lies in
$(\Sym_\pm V)^{\symm_{a+b}[G],\sgn}$).
Since $\symm_{a+b}[G]= \symm_{a + b} \ltimes G^{a+b}$
is generated by its subgroups $\symm_{a+b}$ and $G^{a+b}$, it only remains to
check that any element
$A \shuffle B$ or $A \shuffle^{\sgn} B$ is $G^{a+b}$-fixed.  Since $A$ is $G^a$-fixed and $B$ is $G^b$-fixed, the product $\mu_{a,b}(A \otimes B)$ is $G^{a+b}$-fixed,
and hence remains so after applying
either of the elements $\sh_{a,b}$ or $\sh^{\sgn}_{a,b}$ from $\kk \symm_{a+b}$.
\end{proof}

\bibliographystyle{abbrv}
\bibliography{bibliography}

\end{document}